\documentclass{amsart}

\input xy
\xyoption{all}

\usepackage{geometry}
\usepackage{amsmath}
\usepackage{amssymb}
\usepackage{cite}
\usepackage{amsthm}  
\usepackage{tikz}
\usetikzlibrary{er,positioning}

\newtheorem{same}{This should never appear}[section]
\newtheorem{defin}[same]{Definition}

\newtheorem{remark}[same]{Remark}
\newtheorem{theorem}[same]{Theorem}

\newtheorem{lemma}[same]{Lemma}
\newtheorem{fact}[same]{Fact}

\newtheorem{cor}[same]{Corollary}
\newtheorem{prop}[same]{Proposition}

\newtheorem{hypothesis}[same]{Hypothesis}
\newtheorem{nota}[same]{Notation}

\newtheorem{defin*}{Definition}
\newtheorem*{theorem*}{Theorem}

\makeatletter
\newcommand{\skipitems}[1]{%
  \addtocounter{\@enumctr}{#1}%
}
\newcommand{\bb}{\mathbf{b}}
\newcommand{\rest}{\mathord{\upharpoonright}}
\newcommand{\id}{\textrm{id}}

\newcommand{\K}{\mathbf{K}}

\newcommand{\Ka}{\K^{T}}

\newcommand{\LS}{\operatorname{LS}}

\newcommand{\leap}[1]{\le_{#1}}

\newcommand{\lea}{\leap{\K}}

\newcommand{\gtp}{\mathbf{gtp}}
\newcommand{\gS}{\mathbf{gS}}

\newcommand{\tupop}{\textup{(}}
\newcommand{\tupcp}{\textup{)}}

\DeclareMathOperator{\cof}{cf}    

\title{Superstability, noetherian rings and pure-semisimple rings}
\date{\today.} 

\author{Marcos Mazari-Armida}
\email{mmazaria@andrew.cmu.edu}
\urladdr{http://www.math.cmu.edu/~mmazaria/ }
\address{Department of Mathematical Sciences \\ Carnegie Mellon
University \\ Pittsburgh, Pennsylvania, USA}

\begin{document}

\begin{abstract}

We uncover a connection between the model-theoretic notion of superstability and that of noetherian rings and pure-semisimple rings.

We characterize noetherian rings via superstability  of the class of left modules with embeddings.

\begin{theorem}
For a ring $R$ the following are equivalent.
\begin{enumerate}\item $R$ is left noetherian.
\item The class of left $R$-modules with embeddings is superstable.
\item For every $\lambda \geq |R| + \aleph_0$, there is $\chi \geq \lambda$ such that the class of left $R$-modules with embeddings  has uniqueness of limit models of cardinality $\chi$.
 \item Every limit model in the class of left $R$-modules with embeddings is $\Sigma$-injective. 
\end{enumerate}

\end{theorem}

We characterize left pure-semisimple rings via superstability of the class of left modules with pure embeddings.

\begin{theorem}
For a ring $R$ the following are equivalent.
\begin{enumerate}
\item $R$ is left pure-semisimple.
\item The class of left $R$-modules with pure embeddings is superstable.
\item There exists $\lambda \geq (|R| + \aleph_0)^+$ such that the class of left $R$-modules with pure embeddings has uniqueness of limit models of cardinality $\lambda$.
\item Every limit model in the class of left $R$-modules with pure embeddings is $\Sigma$-pure-injective. 
\end{enumerate}

\end{theorem}

Both equivalences provide evidence that the notion of superstability could shed light in the understanding of algebraic concepts.

As this paper is aimed at model theorists and algebraists an effort
was made to provide the background for both.
\end{abstract}


\maketitle

{\let\thefootnote\relax\footnote{{AMS 2010 Subject Classification:
Primary: 03C48, 16B70. Secondary: 03C45, 03C60, 13L05, 16P40, 16D10.
}}}

\tableofcontents

\section{Introduction}


 An abstract elementary class (AEC) is a pair $\K=(K, \lea)$, where $K$ is class of structures and $\lea$ is a strong substructure relation extending the substructure relation. Among the requirements we have that an AEC is closed under directed colimits and that every set is contained in a small model in the class (see Definition \ref{aec-def}). These were introduced by Shelah in \cite{sh88}. Natural examples in the context of algebra are abelian groups with embeddings, torsion-free abelian groups with pure embeddings, first-order axiomatizable classes of modules with pure embeddings and flat modules with pure embeddings.\footnote{The first two classes were studied in \cite{grp}, \cite{baldwine} and \cite{maz}, the next ones were studied in \cite{kuma} and the last one was studied in \cite{lrv} and \cite{maz1}.}

Dividing lines in complete first-order theories were introduced by Shelah in the late sixties and early seventies. One of the best behaved classes is that of \emph{superstable} theories. Extensions of superstability in a non-elementary setting were first considered in \cite{grsh}. In the context of AECs, superstability was introduced in \cite{sh394} and until recently it was believed to suffer from ``schizophrenia" \cite[p. 19]{shelahaecbook}.  In \cite[1.3]{grva} and \cite{vaseyt}, it was shown (under additional hypothesis that are satisfied by the classes studied in this paper\footnote{The hypothesis are amalgamation, joint embedding, no maximal models and $\LS(\K)$-tameness.})  that superstability is a well-behaved concept and many conditions that were believed to characterize superstability were found to be equivalent. Based on this and the key role \emph{limit models} play in this paper, we will say that an AEC is \emph{superstable} if it has uniqueness of limit models in a tail of cardinals.\footnote{For a complete first-order theory $T$, $(Mod(T), \preceq)$ is superstable if and only if $T$ is superstable as a first-order theory, i.e., $T$ is $\lambda$-stable as for every $\lambda \geq 2^{|T|}$.} Intuitively the reader can think of limit models as universal models with some level of homogeneity (see Definition \ref{limit}).  Further details on the development of the notion of superstability can be consulted in the introduction of \cite{grva}.

 In this paper, we show that the notion of superstability has algebraic substance if one chooses the right context. More specifically, we characterize noetherian rings and pure-semisimple rings via superstability in certain classes of modules with embeddings and with pure embeddings respectively.

\emph{Noetherian rings} are rings with the ascending chain condition for ideals. The precise equivalence we obtain is the following\footnote{Conditions (4) through (6) of the theorem below were motivated by \cite[1.3]{grva}.}.

\textbf{Theorem \ref{main2}.} \textit{For a ring $R$ the following are equivalent.
\begin{enumerate}
\item $R$ is left noetherian.
\item The class of left $R$-modules with embeddings is superstable.
\item For every $\lambda \geq |R| + \aleph_0$, there is $\chi \geq \lambda$ such that the class of left $R$-modules with embeddings  has uniqueness of limit models of cardinality $\chi$.
\item   For every $\lambda \geq |R| + \aleph_0$,  the class of left $R$-modules with embeddings  has uniqueness of limit models of cardinality $\lambda$.
\item   For every $\lambda \geq (|R| + \aleph_0)^+$,  the class of left $R$-modules with embeddings has a superlimit of cardinality $\lambda$.
 \item For every $\lambda \geq |R| + \aleph_0$,  the class of left $R$-modules with embeddings is stable.
 \item Every limit model in the class of left $R$-modules with embeddings is $\Sigma$-injective. 
\end{enumerate}}
\textit{If $R$ is left coherent, they are further equivalent to:
\begin{enumerate}
\skipitems{7}
\item (Lemma \ref{absol}) The class of left absolutely pure modules with embeddings is superstable.
\end{enumerate}}

 A ring $R$ is \emph{left pure-semisimple} if every left $R$-module is pure-injective. It was pointed out to us by Daniel Simson that pure-semisimple rings were introduced by him in \cite{simson77}. There are many papers where several characterizations of pure-semisimple rings are obtained, for example  \cite{cha}, \cite{auslander1}, \cite{auslander}, \cite{zim}, \cite{simson81} and \cite{prest84}. For additional information on what is known about pure-semisimple rings the reader can consult \cite{simson00}, \cite{hui} and \cite[\S 4.5.1]{prest09}. 

In this paper, we give several new characterizations of left pure-semimple rings via superstability. More precisely, we show\footnote{Conditions (4) through (7) of the theorem below were motivated by \cite[1.3]{grva}.}.

\textbf{Theorem \ref{main}.}
\textit{For a ring $R$ the following are equivalent.
\begin{enumerate}
\item $R$ is left pure-semisimple.
\item The class of left $R$-modules with pure embeddings is superstable.
\item There exists $\lambda \geq (|R| + \aleph_0)^+$ such that the class of left $R$-modules with pure embeddings has uniqueness of limit models of cardinality $\lambda$.
\item  For every $\lambda \geq|R| + \aleph_0$,  the class of left $R$-modules with pure embeddings  has uniqueness of limit models of cardinality $\lambda$.
\item For every $\lambda \geq (|R| + \aleph_0)^+$,  the class of left $R$-modules with pure embeddings  has a superlimit of cardinality $\lambda$.
\item  For every $\lambda \geq |R| + \aleph_0$,  the class of left $R$-modules with pure embeddings  is $\lambda$-stable.
\item  For every $\lambda \geq (|R| + \aleph_0)^+$, an increasing chain of $\lambda$-saturated models is $\lambda$-saturated in the class of left $R$-modules with pure embeddings.
\item There exists $\lambda \geq  (|R| + \aleph_0)^+$ such that the class of left $R$-modules with pure embeddings has a $\Sigma$-pure-injective universal model of cardinality $\lambda$.
\item Every limit model in the class of left $R$-modules with pure embeddings is $\Sigma$-pure-injective. 
\end{enumerate}}

A key difference between our results and those of \cite[1.3]{grva} is that in \cite{grva} the cardinal where the \emph{nice property} starts to show up is eventual (bounded by $\beth_{(2^{ |R|+\aleph_0})^+}$), while in our case the cardinal is exactly $|R|+\aleph_0$ or $(|R|+\aleph_0)^+$. In the introduction of \cite{grva} is asked if it was  possible to lower these bounds (see Theorem \ref{equivalent} and the remark below it). 

Although the results of Theorem \ref{main2} and Theorem \ref{main} are similar, the techniques used to prove the results differ significantly. The proof of Theorem \ref{main2} is more rudimentary and depends heavily on the fact that we are working with the class of all modules. On the other hand, Theorem \ref{main} is a corollary of the theory of superstable classes of modules with pure embeddings closed under direct sums which is developed in the fourth section of the paper. One could give a proof of Theorem \ref{main} similar to that of  Theorem \ref{main2}, but we think that the theory of superstable classes of modules with pure embeddings is an interesting theory that should be developed.

Another result of the paper is a positive solution above $\LS(\K)^+$ to Conjecture 2 of \cite{bovan} in the case of classes of modules axiomatizable in first-order with joint embedding and amalgamation (see Theorem \ref{lim-big}). This also provides a partial solution to Question 4.12 of \cite{kuma}.

Algebraically the key idea is to identify limit models with well-understood classes of modules. First, we show that \textit{long} limit models in the class of modules with  embeddings are injective modules (see Lemma 3.7) and that \textit{long} limit models in the class of modules with pure embeddings are pure-injective modules (see Fact \ref{bigpi}) for arbitrary rings. Then by assuming that the ring is noetherian or pure-semisimple we show that all limit models are $\Sigma$-injective (see Theorem \ref{main2}.(7)) or $\Sigma$-pure-injective (see Theorem \ref{main}.(9)) respectively. From these characterizations one can obtain the equivalence with superstability.

The paper is divided into four sections. Section 2 presents necessary background. Section 3  provides a new characterization of noetherian rings. Section 4  characterizes superstability with pure embeddings in classes of modules closed under direct sums and provides a new characterization of pure-semisimple rings. Moreover, a positive solution above $\LS(\K)^+$ to Conjecture 2 of \cite{bovan} is given for certain classes of modules.

It was pointed out to us by Vasey that already in \cite{shlaz} Shelah noticed some connections between superstability of the theory of modules, noetherian rings and pure-semisimple rings. Regarding noetherian rings, Shelah has a remark on page 299 immediately after an unproven theorem (Theorem 8.6) that indicates that he knew that superstability of the theory of modules implies that the ring is left noetherian. The precise equivalence he noticed is similar to that of (6) implies (1) of Theorem \ref{main2}, but the equivalence of (1) and (6) in Theorem 3.12 is new. As for pure-semisimple rings, Shelah claims in Theorem 8.7 that superstability of the theory of modules is equivalent to the ring being pure-semisimple (without mentioning pure-semisimple rings). Shelah fails to provide a proof that if a ring $R$ is pure-semisimple then the theory of $R$-modules is superstable  ((2) to (1) of his Theorem 8.7).  The precise equivalence he noticed is similar to that of (1) and (6) of Theorem \ref{main}.  As the theory of modules is not a complete first-order theory, it is unclear the precise notion of superstability that Shelah refers to in his paper, but he seems to be working with syntactic superstability with respect to positive primitive formulas.

This paper was written while the author was working on a Ph.D. under the direction of Rami Grossberg at Carnegie Mellon University and I would like to thank Professor Grossberg for his guidance and assistance in my research in general and in this work in particular. After reading a preliminary preprint, Sebastien Vasey informed us that he independently discovered the equivalence between superstability and noetherian rings (the equivalence between (1) and (2) of Theorem \ref{main2}), but has not circulated it yet. His proof follows from \cite[3.7]{vaseyu}, \cite[Proposition 3]{eklof} and \cite[5.9]{vaseyt}.  I thank an anonymous referee for comments on another paper of mine that prompted the development of Subsection 4.4. I would also like to thank John T. Baldwin, Daniel Simson, Sebastien Vasey and a couple of referees for comments that helped improve the paper. Finally, I would like to dedicate this work to Marquititos, you will always be loved and remembered.

\section{Preliminaries}

We present the basic concepts of abstract elementary classes that are used in this paper. These are further studied in \cite[\S 4 - 8]{baldwinbook09} and  \cite[\S 2, \S 4.4]{ramibook}. An introduction from an algebraic perspective is given in \cite[\S 2]{maztor}. Regarding the background on module theory, we give a brief survey of the concepts we will use in this paper. An excellent resource for the module theory we will use in this paper are \cite{prest} and \cite{prest09}.

\subsection{Basic concepts}
Abstract elementary classes (AECs) were introduced by Shelah in
\cite[1.2]{sh88}. Among the requirements we have that an AEC is closed under directed colimits and that every set is contained in a small model in the class. Given a model $M$, we
will write $|M|$ for its underlying set and $\| M \|$ for its
cardinality.

\begin{defin}\label{aec-def}
  An \emph{abstract elementary class} is a pair $\K = (K, \lea)$,
where:

  \begin{enumerate}
    \item $K$ is a class of $\tau$-structures, for some fixed
language $\tau = \tau (\K)$. 
    \item $\lea$ is a partial order on $K$. 
    \item $(K, \lea)$ respects isomorphisms: 
    
    If $M \lea N$ are in $K$
and $f: N \cong N'$, then $f[M] \lea N'$. 

In particular 
\tupop taking $M =
N$\tupcp, $K$ is closed under isomorphisms.
    \item If $M \lea N$, then $M \subseteq N$. 
    \item Coherence: If $M_0, M_1, M_2 \in K$ satisfy $M_0 \lea M_2$,
$M_1 \lea M_2$, and $M_0 \subseteq M_1$, then $M_0 \lea M_1$.
    \item Tarski-Vaught axioms: Suppose $\delta$ is a limit ordinal
and $\{ M_i \in K : i < \delta \}$ is an increasing chain. Then:

        \begin{enumerate}

            \item $M_\delta := \bigcup_{i < \delta} M_i \in K$ and
$M_i \lea M_\delta$ for every $i < \delta$.
            \item\label{smoothness-axiom}Smoothness: If there is some
$N \in K$ so that for all $i < \delta$ we have $M_i \lea N$, then we
also have $M_\delta \lea N$.

        \end{enumerate}

    \item L\"{o}wenheim-Skolem-Tarski axiom: There exists a cardinal
$\lambda \ge |\tau(\K)| + \aleph_0$ such that for any $M \in K$ and
$A \subseteq |M|$, there is some $M_0 \lea M$ such that $A \subseteq
|M_0|$ and $\|M_0\| \le |A| + \lambda$. We write $\LS (\K)$ for the
minimal such cardinal.
  \end{enumerate}
\end{defin}

\begin{nota}\
\begin{itemize}

\item If $\lambda$ is a cardinal and $\K$ is an AEC, then $\K_{\lambda}=\{ M \in \K : \| M \|=\lambda \}$.

\item Let $M, N \in \K$. If we write ``$f: M \to N$" we assume that
$f$ is a $\K$-embedding, i.e., $f: M \cong f[M]$ and $f[M] \lea N$.
In particular $\K$-embeddings are always monomorphisms.

\item  Let $M, N \in \K$ and $A \subseteq M$.  If we write ``$f : M \xrightarrow[A]{} N$"  we assume that
$f$ is a $\K$-embedding and that $f\rest_A=\id_{A}$.
\end{itemize}
\end{nota}

  In \cite{sh300} Shelah introduced a notion of semantic type. The
original definition was refined and extended by many authors who
following \cite{grossberg2002} call
these semantic types Galois-types (Shelah recently named them orbital
types).
We present here the modern definition and call them Galois-types
throughout the text. We follow the notation of \cite[2.5]{mv}.

\begin{defin}\label{gtp-def}
  Let $\K$ be an AEC.
  
  \begin{enumerate}
    \item Let $\K^3$ be the set of triples of the form $(\bb, A, N)$, where $N \in \K$, $A \subseteq |N|$, and $\bb$ is a sequence of elements from $N$. 
    \item For $(\bb_1, A_1, N_1), (\bb_2, A_2, N_2) \in \K^3$, we say $(\bb_1, A_1, N_1)E_{\text{at}} (\bb_2, A_2, N_2)$ if $A := A_1 = A_2$, and there exists $f_\ell : N_\ell \xrightarrow[A]{} N$ for $\ell \in\{1, 2\}$ such that $f_1 (\bb_1) = f_2 (\bb_2)$.
    \item Note that $E_{\text{at}}$ is a symmetric and reflexive relation on $\K^3$. We let $E$ be the transitive closure of $E_{\text{at}}$.
    \item For $(\bb, A, N) \in \K^3$, let $\gtp_{\K} (\bb / A; N) := [(\bb, A, N)]_E$. We call such an equivalence class a \emph{Galois-type}. Usually, $\K$ will be clear from context and we will omit it.
\item For $M \in \K$, $\gS_{\K}(M)= \{  \gtp_{\K}(b / M; N) : M
\leq_{\K} N\in \K \text{ and } b \in N\} $.
\item For $\gtp_{\K} (\bb / A; N)$ and $C \subseteq A$, $\gtp_{\K} (\bb / A; N)\upharpoonright_{C}:= [(\bb, C, N)]_E$.

  \end{enumerate}
\end{defin}

\begin{defin} An AEC is \emph{$\lambda$-stable}  if for any $M \in
\K_\lambda$, $| \gS_{\K}(M) | \leq \lambda$. 
\end{defin}

\begin{remark}
 Recall that given $T$ a complete first-order theory and $A \subseteq M$ with $M$ a model of $T$, $S^T(A)$ is the set of complete first-order types with parameters in $A$. For a complete first-order theory $T$ and $\lambda\geq |T|$, $(Mod(T), \preceq)$ is $\lambda$-stable (where $ \preceq$ is the elementary substructure relation) if and only if  $T$ is $\lambda$-stable as a first-order theory, i.e., $|S^T(A)| \leq \lambda$ 
 for every $A \subseteq M$ where $|A|=\lambda$ and $M$ is a model of $T$.
\end{remark}

The following notion was isolated by  Grossberg and VanDieren in \cite{tamenessone}.

\begin{defin} 
$\K$ is \emph{$(< \kappa)$-tame} if for any $M \in \K$ and $p \neq q \in \gS(M)$,  there is $A \subseteq |M|$ such that $|A |< \kappa$ and $p\upharpoonright_{A} \neq q\upharpoonright_{A}$.
\end{defin}

 \subsection{Limit models, saturated models and superlimits} Before introducing the concept of limit model we recall the concept of universal extension.

\begin{defin}
$M$ is \emph{universal over} $N$ if and only if $N \lea M$, $\| M \|=\| N\| =\lambda$ and for any $N^* \in \K_\lambda$ such that $N \lea N^*$, there is $f: N^* \xrightarrow[N]{} M$. 
\end{defin}

With this we are ready to introduce limit models, they were originally introduced in \cite{kosh}.

\begin{defin}\label{limit}
Let $\lambda$ be an infinite cardinal and $\alpha < \lambda^+$ be a limit ordinal.  $M$ is a \emph{$(\lambda,
\alpha)$-limit model over} $N$ if and only if there is $\{ M_i : i <
\alpha\}\subseteq \K_\lambda$ an increasing continuous chain such
that $M_0 :=N$, $M_{i+1}$ is universal over $M_i$ for each $i <
\alpha$ and $M= \bigcup_{i < \alpha} M_i$. We say that $M$ is a $(\lambda, \alpha)$-limit model if there is $N \in
\K_\lambda$ such that $M$ is a $(\lambda, \alpha)$-limit model over
$N$. We say that $M$ is a limit model of cardinality $\lambda$ if there exists a limit ordinal 
$\alpha < \lambda^+$ such that $M$  is a $(\lambda,
\alpha)$-limit model.

\end{defin}

Observe that if $M$ is a $(\lambda, \alpha)$-limit model, then $M$ has cardinality $\lambda$. 

\begin{defin}
Let $\K$ be an AEC and $\lambda$ be a cardinal. $M \in \K$ is a \emph{universal model in
$\K_\lambda$} if $M \in \K_\lambda$ and if given any $N \in \K_\lambda$, there is $f: N \to M$ a $\K$-embedding.
\end{defin}

The following is a simple exercise, a proof is given in \cite[2.10]{maz}.
\begin{fact}\label{univ}
Let $\K$ be an AEC with joint embedding and amalgamation. If $M$ is a limit model of cardinality $\lambda$, then $M$ is a universal model in $\K_\lambda$.
\end{fact}

The next fact gives conditions for the existence of limit models.

\begin{fact}[{\cite[\S II]{shelahaecbook}, \cite[2.9]{tamenessone}}]\label{existence}
Let $\K$ be an AEC with joint embedding, amalgamation and no maximal models. If $\K$ is $\lambda$-stable, then for every $N \in \K_\lambda$ and $\alpha < \lambda^+$ limit ordinal there is $M$ a $(\lambda, \alpha)$-limit model over $N$. Conversely, if $\K$ has a limit model of cardinality $\lambda$, then $\K$ is $\lambda$-stable
\end{fact}

The key question regarding limit models is the uniqueness of limit models of a given cardinality but with chains of different lengths. This has been studied thoroughly in the context of abstract elementary classes \cite{shvi}, \cite{van06}, \cite{grvavi}, \cite{extendingframes}, \cite{vand}, \cite{bovan} and \cite{vasey18}. 

\begin{defin}
$\K$ has \emph{uniqueness of limit models of cardinality $\lambda$} if $\K$ has a limit model of cardinality $\lambda$ and if any two limit models of cardinality $\lambda$ are isomorphic.
\end{defin}

 Since we will only deal with AECs with amalgamation, joint embedding, no maximal models and $\LS(\K)$-tame and it is known (by \cite[1.3]{grva} and \cite{vaseyt}) that in this context the definition below is equivalent to every other definition of superstability considered in the context of AECs, we introduce the following as the definition of superstability.

\begin{defin}
$\K$ is \emph{superstable} if and only if $\K$ has uniqueness of limit models in a tail of cardinals.
\end{defin}

\begin{remark}
For a complete first-order $T$, $(Mod(T), \preceq)$ is superstable if and only if $T$ is superstable as a first-order theory, i.e.,  $T$ is $\lambda$-stable for every $\lambda \geq 2^{|T|}$.
\end{remark} 

 It is important to point out that to establish that $\K$ has uniqueness of limit models of cardinality $\lambda$, one needs to show first the existence of limit models. Due to Fact \ref{existence}, this is equivalent to $\lambda$-stability. 

Another important class of models is that of saturated models.

\begin{defin}
$M \in \K$ is \emph{$\lambda$-saturated} if for every $N \lea M$ and $p \in \gS(N)$ with $\| N \| < \lambda$, there is $a \in M$ such that $p=\gtp(a/N; M)$. $M$ is saturated if $M$ is $\| M \|$-saturated.
\end{defin}

A model $M$ is \emph{$\lambda$-model-homogeneous} if for every $N, N' \in \K$ with $N \lea M$, $N \lea N'$ and $\| N' \| < \lambda$, there is $f: N' \xrightarrow[N]{} M$. Recall that for $\lambda > \LS(\K)$,  a model is $\lambda$-saturated if and only if it is $\lambda$-model-homogeneous. A proof of it appears in \cite[\S II.1.4]{shelahaecbook}.

Superlimit models were introduced in \cite[3.1.(1)]{sh88} as another possible notion of saturation on AECs.

\begin{defin} Let $\K$ be an AEC.
Let $M \in \K$ and $\lambda \geq \LS(\K)$. M is a superlimit in $\lambda$ if:
\begin{enumerate}
\item $M \in \K_\lambda$.
\item For every $N \in \K_\lambda$, there is $f: N \to M$ such that $f[N] \neq M$.
\item If $\{M_i : i < \delta \}\subseteq \K_\lambda$ is an increasing chain, $\delta < \lambda^+$ is an ordinal and $M_i \cong M$ for all $i < \delta$, then $\bigcup_{i < \delta} M_i \cong M$. 
\end{enumerate}
\end{defin}

The following fact has some known connections between limit models, saturated models and superlimits.

\begin{fact}[{\cite[2.8]{grva}, \cite[2.3.10]{dru}}]\label{limits} Let $\K$ be an AEC with amalgamation, joint embedding and no maximal models.
\begin{enumerate}
\item If $\lambda > \LS(\K)$ and $M$ is a $(\lambda, \alpha)$-limit model for $\alpha \in[ \LS(\K)^+, \lambda]$ a regular cardinal, then $M$ is an $\alpha$-saturated model.
\item Let $\lambda > \LS(\K)$ and $\K$ be $\lambda$-stable. $\K$ has uniqueness of limit models in $\lambda$ if and only if every limit model of cardinality $\lambda$ is saturated.
\item Let $\K$ be $\lambda$-stable. If $M$ is a superlimit of cardinality $\lambda$, then $M$ is a $(\lambda, \alpha)$-limit model for every $\alpha < \lambda^+$ limit ordinal. 
\item Let $\lambda > \LS(\K)$, $\K$ be $\lambda$-stable and assume there exists a saturated model of size $\lambda$. $\K$ has a superlimit of cardinality $\lambda$ if and only if the union of an increasing chain (of length less than $\lambda^+$) of saturated models in $\K_\lambda$ is saturated.
\end{enumerate}
\end{fact}

\subsection{Module theory} All rings considered in this paper are associative with an identity element. A module $M$ is \emph{injective} if and only if  for every module
$N$, if $M \leq N$ then $M$ is a direct summand of $N$. 
We say that $M$ is \emph{$\Sigma$-injective} if and only if $M^{(I)}$ is injective for every index set $I$. To consider only countable index sets one needs $M$ to be injective.

\begin{fact}[{\cite[Proposition 3]{faith}}]\label{c-inj} For $M$ an injective module the following are equivalent.
\begin{enumerate}
\item $M$ is $\Sigma$-injective.
\item $M^{(\aleph_0)}$ is injective.
\end{enumerate}
\end{fact}

Recall that a formula $\phi$ is a positive primitive formula ($pp$-formula for short), if $\phi$ is an existentially quantified system of linear equations. Given $M$ and $N$ $R$-modules, $M$ is a \emph{pure submodule} of $N$, denoted by $M \leq_{pp} N$, if and only if $M$ is a submodule of $N$  and for every $pp$-formula $\phi$ it holds that $\phi[N] \cap M = \phi[M]$. Equivalently if for every $L$ right $R$-module $L \otimes M \to L \otimes N$ is a monomorphism.

($\Sigma$-)Pure-injective modules generalize the notion of ($\Sigma$-)injective modules.
A module $M$ is \emph{pure-injective} if in the definition of injective module one substitutes ``$\leq$" by ``$\leq_{pp}$". A module $M$ is  \emph{$\Sigma$-pure-injective} if in the definition of $\Sigma$-injective module one substitutes ``injective" for ``pure-injective". In the case of $\Sigma$-pure-injectivity it is enough to consider countable index sets.

\begin{fact}[{\cite[3.4]{zimm}}]
$M$ is \emph{$\Sigma$-pure-injective} if and only if $M^{(\aleph_0)}$ is pure-injective. 
\end{fact}

A module $M$ is \emph{absolutely pure} if every extension of $M$ is pure. The next fact relates $\Sigma$-injectivity and $\Sigma$-pure-injectivity.

\begin{fact}[{\cite[4.4.16]{prest09}}]\label{sigma-n} For $M$ an $R$-module the following are equivalent.
\begin{enumerate}
\item $M$ is $\Sigma$-injective.
\item $M$ is absolutely pure and $\Sigma$-pure-injective.
\end{enumerate}
\end{fact}

Using the equivalence between $\Sigma$-pure-injectivity and the descending chain condition on $pp$-definable subgroups one can show the following (see for example \cite[2.11]{prest}).

\begin{fact}\label{pp-inj}\
\begin{itemize}
\item If $N$ is $\Sigma$-pure-injective and $M \leq_{pp} N$, then $M$ is $\Sigma$-pure-injective.
\item If $N$ is $\Sigma$-pure-injective and $M$ is elementary equivalent to $N$, then $M$ is $\Sigma$-pure-injective.
\end{itemize}
\end{fact}

We will also use that $\Sigma$-pure-injective modules are totally transcendental.

\begin{fact}[{\cite[3.2]{prest}}]\label{sigma-stability}
If $M$ is $\Sigma$-pure-injective, then $(Mod(Th(M)), \preceq)$ is $\lambda$-stable for every $\lambda \geq |Th(M)|$.
\end{fact}

 A ring $R$ is \emph{left noetherian} if every increasing chain of left ideals is stationary. These were introduced by Noether in \cite{noe}. Following \cite{eklof}, denote by $\gamma_R$ the smallest cardinal such that every left ideal of $R$ is generated by less than $\gamma_R$ elements. Observe that $\gamma_R \leq |R|^+$.  We will use the following equivalence later in the paper. The equivalence between one and four is due to Cartan-Eilenberg-Bass-Papp and the equivalence between one and two is trivial

\begin{fact}[{\cite[4.4.17]{prest09}}]\label{equivnoe} For a ring $R$ the following are equivalent.
\begin{enumerate}
\item $R$ is left noetherian.
\item $\gamma_R \leq \aleph_0$.
\item Every injective left $R$-module is $\Sigma$-injective.
\item Every direct sum of injective left $R$-modules is injective.
\item Every absolutely pure left $R$-module is injective.
\end{enumerate}
\end{fact}

Recall the notion of a left pure-semisimple ring. 

\begin{defin}\label{pss}
A ring $R$ is \emph{left pure-semisimple} if and only if every left $R$-module $M$ is pure-injective.
\end{defin}

Many equivalent conditions have been found for the notion of a pure-semisimple ring, see for example  \cite{cha}, \cite{auslander1}, \cite{auslander}, \cite{zim}, \cite{simson81} and \cite{prest84}. A more updated set of equivalences is given in \cite{simson00} and \cite[\S 4.5.1]{prest09}. Below we give some of the equivalent conditions for a ring to be left pure-semisimple.

\begin{fact}[{\cite[11.3]{prest}}]\label{pps} For a ring $R$ the following are equivalent.
\begin{enumerate}
\item $R$ is left pure-semisimple.
\item Every left $R$-module $M$ is $\Sigma$-pure-injective.
\item Every left $R$-module is the direct sum of indecomposable submodules.
\end{enumerate}

\end{fact}

Recall Bumby's result \cite{bumby} and its generalization to pure-injective modules. A proof of both results (and a discussion of 
the general setting) appears in
\cite[2.5]{gks}.

\begin{fact}\
\label{ipi}\begin{itemize}
\item Let $M, N$ be injective modules. If there are $f: M \to N$ an embedding and $g: N \to M$ an 
embedding, then $M$ is isomorphic to $N$. 
\item Let $M, N$ be pure-injective modules. If there are $f: M \to N$ a pure embedding and $g: N \to M$ a 
pure embedding, then $M$ is isomorphic to $N$. 
\end{itemize}
\end{fact}

\subsection{Notation}
We will use the following notation which was introduced in \cite[3.1]{kuma}.

\begin{nota}\label{not} Given $R$ a ring, we denote by
$\textbf{Th}_R$ the theory of left $R$-modules. A (not necessarily complete) first-order theory $T$ is \emph{a theory of modules} if it extends $\textbf{Th}_R$. For $T$ a theory of modules, let $\K^{T}= ( Mod(T),
\leq_{pp})$
and $|T|=|R|+\aleph_0$.

\end{nota}

Since we will also work with embeddings we introduce the following notation.

\begin{nota}
Given $R$ a ring, we will use the standard notation $(R\text{-Mod}, \subseteq_R)$ instead of the model-theoretic notation $(Mod(\textbf{Th}_R), \leq)$ to denote the AEC of left $R$-modules with embeddings.
\end{nota}

\section{A new characterization of noetherian rings}

In this section we will work in the class of modules with embeddings. Since complete theories of modules only have $pp$-quantifier elimination, we do not think that in the case of classes of modules with embeddings there is a deep theory as the one we will develop in the next section for pure embeddings. Instead, using some more rudimentary methods, we will study the class of modules with embeddings.

\begin{remark}
It is well-known that $(R\text{-Mod}, \subseteq_R)$ is an AEC that has amalgamation, joint embedding and no maximal models. 
\end{remark}

The next assertion describes Galois-types in this context.

\begin{lemma}
Let $M, N_1, N_2 \in R\text{-Mod}$,  $M \subseteq_R N_1, N_2$, $\bar{b}_{1}
\in  N_1^{<\omega}$  and $\bar{b}_{2} \in N_2^{<\omega}$. Then:
 \[ \gtp_{(R\text{-Mod}, \subseteq_R)}(\bar{b}_{1}/M; N_1) = \gtp_{(R\text{-Mod}, \subseteq_R)}(\bar{b}_{2}/M; N_2) \text{ if
and
only if } qf\text{-}tp(\bar{b}_{1}/M , N_1) = qf\text{-}tp(\bar{b}_{2}/M, N_2).\]
\end{lemma}
\begin{proof}[Proof sketch]
The forward direction is trivial, so let us sketch the backward direction. By the amalgamation property we may assume that $N_1=N_2=N$. Define $f: \langle \bar{b}_1 M \rangle \to \langle \bar{b}_2 M \rangle$ as $f(\Sigma_{i=1}^{n} r_ib_{1,i} + \Sigma_{i=1}^{k} s_i m_i)= \Sigma_{i=1}^{n} r_ib_{2,i} + \Sigma_{i=1}^{k} s_i m_i$ where $\langle \bar{b}_\ell M \rangle$ is the submodule generated by $\bar{b}_\ell M$ inside $N$ for $\ell \in \{1, 2\}$, $r_i, s_i \in R$ for all $i$ and $m_i \in M$ for all $i$. Using that the quantifier free types are equal, it follows that $f$ is an isomorphism. Then the result follows by applying amalgamation a couple of times. \end{proof}

Since we can witness that two Galois-types are different by a quantifier free formula, we obtain.

\begin{cor}\label{tame-mod}
$(R\text{-Mod}, \subseteq_R)$ is $(< \aleph_0)$-tame.
\end{cor}

The above corollary also follows from the general theory of AECs \cite[3.7]{vaseyu}, since $(R\text{-Mod}, \subseteq_R)$ is a universal class in the sense of  \cite{tarski} (see \cite[2.1]{mv} for the definition).

An analogous argument to the one given in \cite[4.8]{kuma} can be used to show the following.

\begin{prop}\label{easyp} Let $\lambda$ be an infinite cardinal.
If $E \in (R\text{-Mod}, \subseteq_R)_\lambda $ is injective and $U \in R\text{-Mod}$ is universal in $(R\text{-Mod}, \subseteq_R)_\lambda$, then $E \oplus U$ is universal over $E$. 
\end{prop}

The next fact from \cite{eklof} will be useful.

\begin{fact}[ {\cite[Proposition 3]{eklof}}]\label{uni-mod}
Let $\lambda$ be an infinite cardinal with $\lambda \geq |R| + \aleph_0$. $\lambda^{<\gamma_R}=\lambda$ if and only if there is an injective universal model in $(R\text{-Mod}, \subseteq_R)_\lambda$. 
\end{fact}

With it we will be able to show that $(R\text{-Mod}, \subseteq_R)$ is stable.

\begin{lemma}\label{sta-m}
Let $R$ be a ring and $\lambda$ be an infinite cardinal with $\lambda \geq |R| + \aleph_0$. If $\lambda^{<\gamma_R}=\lambda$, then $(R\text{-Mod}, \subseteq_R)$ is $\lambda$-stable.
\end{lemma}
\begin{proof}
By Fact \ref{uni-mod} there is $U$ an injective universal model in $(R\text{-Mod}, \subseteq_R)_\lambda$. Build $\{ N_i : i < \omega \}$ by induction such that $N_i$ is equal to $(i+1)$-many direct copies of $U$.

Since $U$ is injective and injective objects are closed under finite direct sums, it follows that $N_i$ is injective for every $i < \omega$. Moreover, by Proposition \ref{easyp} it follows that $N_{i+1}$ is universal over $N_i$ for every $i < \omega$. Let $N= \bigcup_{i<\omega} N_i$. Observe that $N$ is a $(\lambda, \omega)$-limit model, so by Fact \ref{existence} it follows that $(R\text{-Mod}, \subseteq_R)$ is $\lambda$-stable. \end{proof}

From the above theorem and Fact \ref{existence} it follows that there is a $(\lambda, \alpha)$-limit model for every $\alpha < \lambda^+$ limit ordinal and cardinal $\lambda$ such that $\lambda^{<\gamma_R}=\lambda$. The next lemma characterizes limit models in $(R\text{-Mod}, \subseteq_R)$.

\begin{lemma}\label{lim-inj} Let $R$ be a ring, $\lambda$ be an infinite cardinal with $\lambda \geq \gamma_R + |R| + \aleph_0$ and $\alpha < \lambda^+$ be a limit ordinal.
If $M$ is a $(\lambda, \alpha)$-limit model in $(R\text{-Mod}, \subseteq_R)$ and $\cof(\alpha)\geq \gamma_R$, then $M$ is injective.
\end{lemma}
\begin{proof}
By \cite[Lemma 2]{eklof} it is enough to show that if $\mathbb{E}=\{ r_{\delta} x= a_\delta : \delta < \beta \}$  is a system of equations in one free variable $x$ with $\beta < \gamma_R$ and $r_\delta \in R$, $a_\delta \in M$ for every $\delta < \beta$ and $\mathbb{E}$  has a solution in an extension of $M$, then $\mathbb{E}$ has a solution in $M$. 

Let $\mathbb{E}$ be a system of equations as in the previous paragraph, $M'\in (R\text{-Mod}, \subseteq_R)_\lambda$ be an extension of $M$ with $b \in M'$ realizing $\mathbb{E}$ and $\{ M_i : i < \alpha \}$ be a witness to the fact that $M$ is a $(\lambda, \alpha)$-limit model. Since $\beta < \gamma_R$ and $\cof(\alpha)\geq \gamma_R$, there is $i < \alpha$ such that $\{a_\delta : \delta < \beta  \} \subseteq M_i$. Since $M_{i+1}$ is universal over $M_i$ there is $f: M' \xrightarrow[M_i]{} M$. It is clear that $f(b) \in M$ realizes $\mathbb{E}$. \end{proof}

Using the above lemma, we can obtain an equivalence in Lemma \ref{sta-m}

\begin{cor}\label{st-no}
Let $R$ be a ring and $\lambda$ be an infinite cardinal with $\lambda \geq (|R| + \aleph_0)^+$.  $\lambda^{<\gamma_R}=\lambda$ if and only if $(R\text{-Mod}, \subseteq_R)$ is $\lambda$-stable.
\end{cor}
\begin{proof}
The forward direction is Lemma \ref{sta-m} and the backward direction follows from the existence of limit models, Lemma \ref{lim-inj} and Fact \ref{uni-mod}.
\end{proof}

Doing a similar proof to that of Lemma \ref{lim-inj} and using the equivalence between saturation and model-homogeneity one can obtain the next result.

\begin{lemma}\label{gst-inj}
Let $\lambda \geq (|R| + \aleph_0)^+$. If $M$ is $\lambda$-saturated in $(R\text{-Mod}, \subseteq_R)$, then $M$ is injective.
\end{lemma}

Since a ring $R$ is noetherian if and only if $\gamma_R \leq \aleph_0$ (by Fact \ref{equivnoe}), the next result follows from the results we just obtained in this section.

\begin{cor}\label{noe}
If $R$ is a left noetherian ring, then:
\begin{enumerate}
\item $(R\text{-Mod}, \subseteq_R)$ is $\lambda$-stable for every $\lambda \geq |R| + \aleph_0$.
\item There is a $(\lambda, \alpha)$-limit model in $(R\text{-Mod}, \subseteq_R)$ for every $\lambda \geq |R| + \aleph_0$ and $\alpha < \lambda^+$ limit ordinal.
\item Every limit model in $(R\text{-Mod}, \subseteq_R)$ is injective.
\end{enumerate}
\end{cor}

Moreover, the analogous of \cite[4.9]{kuma} can also be carried out in this context. Since the proof of the proposition is basically the same as that of \cite[4.9]{kuma} we omit it.

\begin{prop}\label{count-n}
Assume $\lambda\geq
(|R| + \aleph_0)^+$. If $M$ is
a $(\lambda, \omega)$-limit model in $(R\text{-Mod}, \subseteq_R)$ and  $N$ is a
$(\lambda,(|R| + \aleph_0)^+)$-limit model in $(R\text{-Mod}, \subseteq_R)$, then $M$ is isomorphic to $ N^{(\aleph_0)}$.
\end{prop}

With this we obtain a new characterization of left noetherian rings via superstability.\footnote{Conditions (4) through (6) of the theorem below were motivated by \cite[1.3]{grva}.}

\begin{theorem}\label{main2} For a ring $R$ the following are equivalent.
\begin{enumerate}
\item $R$ is left noetherian.
\item The class of left $R$-modules with embeddings is superstable.
\item For every $\lambda \geq |R| + \aleph_0$, there is $\chi \geq \lambda$ such that the class of left $R$-modules with embeddings  has uniqueness of limit models of cardinality $\chi$.
\item   For every $\lambda \geq |R| + \aleph_0$,  the class of left $R$-modules with embeddings  has uniqueness of limit models of cardinality $\lambda$.
\item   For every $\lambda \geq (|R| + \aleph_0)^+$,  the class of left $R$-modules with embeddings has a superlimit of cardinality $\lambda$.
 \item For every $\lambda \geq |R| + \aleph_0$,  the class of left $R$-modules with embeddings is $\lambda$-stable.
 \item Every limit model in the class of left $R$-modules with embeddings is $\Sigma$-injective. 
\end{enumerate}
\end{theorem}
\begin{proof}
(1) $\Rightarrow$ (4) Let $\lambda \geq |R| + \aleph_0$. By Corollary \ref{noe}.(2) there is $(\lambda, \alpha)$-limit models for every $\alpha < \lambda^+$ limit ordinal. So we only need to show uniqueness of limit models. Let $M$ and $N$  be two limit models of cardinality $\lambda$. By Corollary \ref{noe}.(3) $M$ and $N$ are injective and since $M$ embeds into $N$ and vice versa by Fact \ref{univ}, it follows from Fact \ref{ipi} that $M$ is isomorphic to $N$.

(4) $\Rightarrow$ (2) and (2) $\Rightarrow$ (3) are clear.

(3) $\Rightarrow$ (1) We will use the equivalence given in Fact \ref{equivnoe}.(3), so let $M$ be an injective module. Since $M$ is injective, it is absolutely pure so by Fact \ref{sigma-n} it is enough to show that $M$ is $\Sigma$-pure-injective. Let $\chi \geq (\| M \| + |R| + \aleph_0)^+$ such that $(R\text{-Mod}, \subseteq_R)$ has uniqueness of limit models of cardinality $\chi$. Let $N$ be a $(\chi, (|R| + \aleph_0)^+)$-limit model such that $M \subseteq_R N$.  By Proposition \ref{count-n} $N^{(\aleph_0)}$ is a  $(\chi, \omega )$-limit model, then by uniqueness of limit models and Lemma \ref{lim-inj} (using that $(|R| + \aleph_0)^+\geq \gamma_R$)  $N^{(\aleph_0)}$ is injective. Then $N$ is $\Sigma$-pure-injective. Since $M$ is injective, it follows that $M$ is $\Sigma$-pure-injective by Fact \ref{pp-inj}.


(1) $\Rightarrow$ (5) Let $\lambda \geq  (|R| + \aleph_0)^+$. By Corollary \ref{noe}.(1) $(R\text{-Mod}, \subseteq_R)$ is $\lambda$-stable, so let $M$ be a  $(\lambda, (|R| + \aleph_0)^+)$-limit model. Then by condition (4) together with Fact \ref{limits}.(2) $M$ is saturated. So by Fact \ref{limits}.(4) it is enough to show that an increasing chain of saturated models in $(R\text{-Mod}, \subseteq_R)_\lambda$ is saturated. Let $\{ M_i : i <\delta\}$ be an increasing chain of saturated models in $(R\text{-Mod}, \subseteq_R)_\lambda$ and $\delta < \lambda^+$ be a limit ordinal.

Using that $\gamma_R \leq \aleph_0$ and that $M_i$ is injective for every $i < \delta$ (by Lemma \ref{gst-inj}), one can show that $\bigcup_{i < \delta} M_i$ is an injective module. Moreover, $\bigcup_{i < \delta} M_i$ embeds into $M_0$, because $M_0$ saturated, and $M_0$ embeds into  $\bigcup_{i < \delta} M_i$. Then by Fact \ref{ipi} $M_0$ is isomorphic to $\bigcup_{i < \delta} M_i$. Hence,  $\bigcup_{i < \delta} M_i$ is saturated.

(5) $\Rightarrow$ (2) Let $\lambda \geq |R| + \aleph_0$ and $\chi=(\lambda^+)^{\gamma_R}$. By Lemma \ref{sta-m} $(R\text{-Mod}, \subseteq_R)$ is $\chi$-stable. Therefore there are $(\chi, \alpha)$-limit models for every $\alpha < \chi^+$ limit ordinal. So we only need to show uniqueness of limit models. Let $M$ and $N$ be two limit models of cardinality $\chi$. Let $L$ be a superlimit of size $\chi$, then by Fact \ref{limits}.(3) $M$ is isomorphic to $L$ and $N$ is isomorphic to $L$. Hence $M$ is isomorphic to $N$.

(1) $\Rightarrow$ (6) By Corollary \ref{noe}.(1).

(6) $\Rightarrow$ (1) Assume for the sake of contradiction that $R$ is not noetherian, then it follows that $\gamma_R > \aleph_0$. Let $\lambda = \beth_\omega( |R| + \aleph_0)$. Since $\cof(\lambda)=\omega$, we have by K\"{o}nigs lemma that $\lambda^{<\gamma_R} > \lambda$. Then by Corollary \ref{st-no} it follows that $(R\text{-Mod}, \subseteq_R)$ is not $\lambda$-stable. A contradiction to the hypothesis.

(1) $\Rightarrow$ (7) Let $M$ be a $(\lambda, \alpha)$-limit model for $\alpha < \lambda^+$ limit ordinal. Let $\mu \geq (\| M \| + |R| + \aleph_0)^+$ and $N$ be a $(\mu, (|R| + \aleph_0)^+)$-limit model such that $M\subseteq_R N$, it exists by Corollary \ref{noe}.(2). By Proposition \ref{count-n} $N^{(\aleph_0)}$ is a  $(\mu, \omega )$-limit model, then by Corollary \ref{noe}.(3) $N^{(\aleph_0)}$ is injective. So $N$ is $\Sigma$-pure-injective. Therefore, since $M$ is injective by Corollary \ref{noe}.(3), it follows that $M$ is $\Sigma$-pure-injective by Fact \ref{pp-inj}. Hence $M$ is $\Sigma$-injective by Fact \ref{sigma-n}.

(7) $\Rightarrow$ (3) Let $\lambda \geq |R| + \aleph_0$ and let $\chi \geq \lambda$ such that  $(R\text{-Mod}, \subseteq_R)$ is $\chi$-stable, this exist by Corollary \ref{st-no}. There are clearly limit models of cardinality $\chi$, so we only need to show uniqueness of limit models. Let $M$ and $N$ be two limit models of cardinality $\chi$. By hypothesis $M$ and $N$ are injective and since $M$ embeds into $N$ and vice versa by Fact \ref{univ}, it follows from Fact \ref{ipi} that $M$ is isomorphic to $N$. \end{proof}

This is not the first result where noetherian rings and superstability have been related. As it was mentioned in the introduction,  Shelah noticed that superstability of the theory of modules implies that the ring is left noetherian in \cite[\S 8]{shlaz}. The precise equivalence he noticed is similar to that of (6) implies (1) of the above theorem. For a countable $\omega$-stable ring a proof is given in \cite[9.1]{balmc}. The equivalence between (1) and (6) is new. Another paper that relates both notions is \cite{grsh}. In it, it is shown (for integral domains) that superstability (in the sense of \cite[1.2]{grsh}) of the class of torsion divisible modules implies that the ring is noetherian.

\begin{remark}\label{remk}
Compared to \cite[1.3]{grva}, the above theorem improves the bounds where the \emph{nice propertis} show up from $\beth_{(2^{|R| + \aleph_0})^+}$ to $|R|+\aleph_0$ or $(|R|+\aleph_0)^+$ in the class of modules with embeddings.  In the introduction of \cite{grva} is asked if this bounds can be improved.
\end{remark}

\begin{remark} Since $(R\text{-Mod}, \subseteq_R)$ has amalgamation, joint embedding, no maximal models and is $(< \aleph_0)$-tame (by Corollary \ref{tame-mod}) . Therefore, we could have simply quoted the main theorem of \cite{grva} and \cite{vaseyt} to obtain (5),(6) imply (2) of the above theorem. We decided to provide the proofs for those directions to make the paper more transparent and since the proofs in our case are easier than in the general case. 
\end{remark}

\section{A new characterization of pure-semisimple rings}

It is possible to obtain a similar proof for Theorem \ref{main} (without conditions (7) and (8)) as the one presented for Theorem \ref{main2}. The reason we do not do this is because there is a deep theory when one considers theories of modules with pure embeddings. What we will do is to study superstability for any first-order theory of modules with pure embeddings and as a simple corollary we will obtain Theorem \ref{main}.

 Recall that for $T$ a first-order theory of modules (not necessarily complete) $\K^{T}= ( Mod(T),
\leq_{pp})$. In \cite[3.4]{kuma}, it is shown that if $T$ is a theory of modules then $\K^T$ is an AEC. Most interesting results do not hold for all AECs, so we will assume, as in \cite{kuma}, the next hypothesis throughout this section.

\begin{hypothesis}\label{hyp}
Let $R$ be a ring and $T$ be a theory of modules with an infinite model such that:
\begin{enumerate}
\item $\K^T$ has joint embedding.
\item $\K^T$ has amalgamation. 
\end{enumerate}
\end{hypothesis}

These may seem like \textit{adhoc} hypothesis, but there are many natural theories satisfying them. This is the case if $T$ is a complete theory, but many other examples are given in \cite[3.10]{kuma}. For the proof of the main theorem, we will only use the well-known result that  $\K^{\textbf{Th}_R}=(R\text{-Mod}, \leq_{pp})$ satisfies the above hypothesis.

Since the theory of modules has $pp$-quantifier elimination (see for example  \cite[1.1]{ziegler}), one can show the following.
\begin{fact}[{\cite[3.14]{kuma}}]\label{pp=gtp}
Let $M, N_1, N_2 \in \K^T$,  $M \leq_{pp} N_1, N_2$, $\bar{b}_{1}
\in  N_1^{<\omega}$  and $\bar{b}_{2} \in N_2^{<\omega}$. Then:
 \[ \gtp(\bar{b}_{1}/M; N_1) = \gtp(\bar{b}_{2}/M; N_2) \text{ if
and
only if } pp(\bar{b}_{1}/M , N_1) = pp(\bar{b}_{2}/M, N_2).\]

Moreover, if $N_1 \equiv N_2$ one can substitute the $pp$-types by the first-order types (\cite[3.13]{kuma}).
\end{fact}

\subsection{The theory $\tilde{T}$} In \cite[3.16]{kuma} it is shown that $\K^T$ is $\lambda$-stable if  $\lambda^{|T|}=\lambda$. Then it follows 
from Fact \ref{existence} that there exist limit models of cardinality $\lambda$  in $\K^T$ for every cardinal $\lambda$ such that $\lambda^{|T|}=\lambda$. More importantly and key to the naturality of the theory we will introduce in this section is the following result.

\begin{fact}[{\cite[4.3]{kuma}}]\label{limeq}
If $M$ and $N$ are limit models in $\Ka$, then $M$ and $N$ are elementary equivalent.
\end{fact}

Let us introduce the main notion of this subsection.

\begin{nota}\label{notat}
For $T$ a theory of modules, let $\tilde{M}_T$ be the $(2^{|T|}, \omega)$-limit model of $\K^T$ and $\tilde{T}=Th(\tilde{M}_T)$.
\end{nota}

It is natural to ask which structures of $\Ka$ satisfy the complete first-order theory $\tilde{T}$. It follows from Fact \ref{limeq} that limit models do, we record this for future reference.

\begin{cor}\label{limeqq}
If $M$ is a limit model in $\Ka$, then $M$ is a model of $\tilde{T}$.
\end{cor}

The next lemma gives another class of structures satisfying $\tilde{T}$.

\begin{lemma}\label{gs=limit} Let $\lambda \geq |T|^+$.
If $M$ is a $\lambda$-saturated model in $\Ka$, then $M$ is a model of $\tilde{T}$.
\end{lemma}
\begin{proof}[Proof sketch] 
The proof is similar to that of \cite[4.3]{kuma}, by using the equivalence between $\lambda$-saturation and $\lambda$-model-homogeneity.
\end{proof}

Moreover, $\tilde{T}$ is closed upward under pure extensions.

\begin{lemma}\label{up-close}
If $M$ is a model of $\tilde{T}$, \textbf{$N \in \Ka$} and $M \leq_{pp} N$, then $M \preceq N$. In particular, $N$ is a model of $\tilde{T}$.
\end{lemma}
\begin{proof}
Let $\lambda= \| N\|^{|T|}$, then by \cite[3.16]{kuma} $\Ka$ is $\lambda$-stable. So let $N^* \in \Ka_\lambda$ such that $N \leq_{pp} N^*$ and $N^*$ is a $(\lambda, \omega)$-limit model. By Fact \ref{limeq} $N^* \equiv M$. Then $M \preceq N \preceq N^*$ by \cite[2.25]{prest}. Whence $M \preceq N$ and $N$ is a model of $\tilde{T}$.
\end{proof}

The next result shows the naturality of $\tilde{T}$, the proof is similar to that of the above lemma so we omit it.
\begin{cor}
If $M \in \Ka$, then there is $N$ a model of $\tilde{T}$ such that $M \leq_{pp} N$. Moreover, if $T'$ is a complete first-order theory with this property, then $\tilde{T}=T'$.
\end{cor}

The following lemmas show that there is a close relationship between the class $\Ka$ and the first-order theory $\tilde{T}$. This is useful since complete first-order theories of modules are very well understood (see for instance \cite{prest}). 

\begin{lemma}\label{stable}
For $T$ a theory of modules and $\lambda \geq |T|$, the following are equivalent.
\begin{enumerate}
\item $\tilde{T}$ is $\lambda$-stable. 
\item $\K^T$ is $\lambda$-stable.
\end{enumerate}
\end{lemma}
\begin{proof}
$\Rightarrow :$ Let $M\in\Ka_\lambda$ and $\{ p_i : i < \alpha \}$ be an enumeration without repetitions of $\gS(M)$. Fix $\mu=\lambda^{|T|}$ and $N \in \Ka_\mu$ a $(\mu, \omega)$-limit model such that $M \leq_{pp} N$. Then there is $\{a_i : i < \alpha \} \subseteq N$ such that $p_i= \gtp(a_i/ M; N)$ for every $i < \alpha$.

Let $\Phi: \gS(M) \to S^{Th(N)}(M)$ be defined by $\phi(\gtp(a_i/ M; N))= tp(a_i/M, N)$. By Fact \ref{pp=gtp} it follows that $\Phi$ is a well-defined injective function, so $|\gS(M)| \leq  |S^{Th(N)}(M)|$. Finally, since $N$ is a model of $\tilde{T}$ by Corollary \ref{limeqq} and $\tilde{T}$ is $\lambda$-stable by hypothesis, it follows that $ |S^{Th(N)}(M)| \leq \| M \|$. Hence $|\gS(M)|\leq \| M \|$.

$\Leftarrow :$  Let $A \subseteq N$ of size $\lambda$ with $N \vDash \tilde{T}$ and $N$ is $\lambda^{+}$-saturated in $\tilde{T}$. Let $\{ p_i : i < \alpha \}$ be an enumeration without repetitions of $S^{\tilde{T}}(A)$. Let $\{a_i : i < \alpha \} \subseteq N$ such that $p_i= tp(a_i/ A, N)$ for every $i < \alpha$.

Let $M$ be the structure obtained by applying downward L\"{o}wenheim-Skolem-Tarski to $A$ in $N$, observe that $\|M\| =\lambda$ because $\lambda \geq |T|$. Let $\Psi: S^{\tilde{T}}(A) \to \gS(M)$ be defined by $\phi(tp(a_i/ A, N))= \gtp(a_i/M; N)$. By Fact \ref{pp=gtp} it follows that $\Phi$ is a well-defined injective function and doing a similar argument to the one above it follows that $|S^{\tilde{T}}(A)| \leq  \| M \|=|A|$. \end{proof}

\begin{lemma}\label{st=gst}
For $T$ a theory of modules and $\lambda \geq |T|^+$, the following are equivalent.
\begin{enumerate}
\item $M$ is a model of $\tilde{T}$ and $M$ is $\lambda$-saturated in $\tilde{T}$.
\item M is $\lambda$-saturated in $\Ka$.
\end{enumerate}
\end{lemma}
\begin{proof}
$\Rightarrow$: Let $L \in \Ka$, $L \leq_{pp} M$, $\| L \| < \lambda$ and $p \in \gS(L)$. Let $L^*$ be a $(\|L\|^{|T|}, \omega)$-limit model such that $L \leq_{pp} L^*$ and $a \in L^*$ with $p=\gtp(a/L; L^*)$. 

Realize that $L^*$ is a model of $\tilde{T}$ by Fact \ref{limeq}, so $tp(a/L, L^*) \in S^{\tilde{T}}(L)$. Then since $M$ is $\lambda$-saturated in $\tilde{T}$ there is $b \in M$ such that $ tp(a/L, L^*)=tp(b/L, M)$.  Therefore, since $L^* \equiv M$ and Fact \ref{pp=gtp}, we conclude that $\gtp(a/L; L^*)=  \gtp(b/ L; M)$.

$\Leftarrow$: By Lemma \ref{gs=limit} $M$ is a model of $\tilde{T}$. Let $A \subseteq M$ and $p \in S^{\tilde{T}}(A)$ with $|A| < \lambda$. Let $N$ be an elementary extension of $M$ and $a \in N$ such that $p =tp(a/A, N)$. Let $M^*$ be the structure obtained by applying downward L\"{o}wenheim-Skolem-Tarski to $A$ in $M$, observe that $\|M^*\| < \lambda$ because $\lambda \geq |T|^+$.

Realize that $M^* \leq_{pp} M \leq_{pp}  N$, so $\gtp( a/ M^*; N) \in \gS(M^*)$. Then since $M$ is $\lambda$-saturated in $\Ka$, there is $b \in M$ such that  $\gtp( a/ M^*; N) = \gtp( b/ M^*; M)$. Therefore, since $M \equiv N$ and Fact \ref{pp=gtp}, we conclude that $tp( a/ M^*, N) = tp( b/ M^*, M)$. Hence $b \in M$ realizes $p$. \end{proof}

The following result was pointed out to us by an anonymous referee.

\begin{lemma}\label{lim-til} For $T$ a theory of modules, $\lambda \geq |T|$ and $\alpha < \lambda^+$ a limit ordinal, the following are equivalent. 
\begin{enumerate}
\item $M$ is a $(\lambda, \alpha)$-limit model in $(Mod(\tilde{T}), \preceq)$.
\item $M$ is a $(\lambda, \alpha)$-limit model in $\Ka$.
\end{enumerate}
\end{lemma}
\begin{proof}
$\Rightarrow$: Let $\{M_i : i < \alpha\}$ be a witness to the fact that $M$ is a $(\lambda, \alpha)$-limit model in $(Mod(\tilde{T}), \preceq)$. Observe that $\{M_i : i < \alpha\}$ is chain of models in $\Ka$, so it is enough to show that $M_{i+1}$ is universal over $M_i$ for every $i < \alpha$ in $\Ka$. This follows easily from Lemma \ref{up-close}.

$\Leftarrow$: Let $\{M_i : i < \alpha\}$ be a witness to the fact that $M$ is a $(\lambda, \alpha)$-limit model in $\K^T$. $\Ka$ is $\lambda$-stable, by Fact \ref{existence}, so let $N$ be a $(\lambda, \omega)$-limit model over $M_0$. Since $M_1$ is universal over $M_0$, there is $f: N \xrightarrow[M_0]{} M_1$. Then by Lemma \ref{up-close} it follows that $\{ M_i :  0< i < \alpha\}$ is an elementary chain of models of $\tilde{T}$. Using Lemma  \ref{up-close} once again, one can show that $M_{i+1}$ is universal over $M_i$ for every $i < \alpha$ in $(Mod(\tilde{T}), \preceq)$. Hence $\{ M_i :  0< i < \alpha\}$ is a witness to the fact that $M$ is a $(\lambda, \alpha)$-limit model in $(Mod(\tilde{T}), \preceq)$. \end{proof}

Given $\phi, \psi$ $pp$-formulas in one free
variable such that $\textbf{Th}_R \vdash \psi \to \phi$ and a module $M$, $Inv(M,\phi, \psi)$ is the size of $\phi(M)/ \psi(M)$ if $|\phi(M)/ \psi(M)|$ is finite and infinity otherwise.
\begin{lemma}
Let $T$ be a theory of modules. If $\Ka$ is closed under direct sums, then $\tilde{T}$ is closed under direct sums.
\end{lemma}
\begin{proof}
Recall that $\tilde{M}_T$ is the $(2^{|T|}, \omega)$-limit model of $\K^T$ .

\fbox{Claim} $Inv(\tilde{M}_T,\phi, \psi)=1$ or $\infty$ for every $\phi, \psi$ $pp$-formulas in one free
variable such that $\textbf{Th}_R \vdash \psi \to \phi$.

 \underline{Proof of Claim}: Let  $\phi, \psi$ $pp$-formulas such that $\textbf{Th}_R
\vdash \psi \to \phi$ and assume for the sake of contradiction that $Inv(\tilde{M}_T,\phi, \psi)=k>1$ for $k \in \mathbb{N}$. Since $\Ka$ is closed under direct sums, $\tilde{M}_T \oplus \tilde{M}_T \in \Ka$ and by Fact \ref{univ} there is $f: \tilde{M}_T \oplus \tilde{M}_T \to \tilde{M}_T$ pure embedding.  Then:

\[ k^2= Inv( \tilde{M}_T \oplus \tilde{M}_T, \phi, \psi) = Inv( f[\tilde{M}_T \oplus \tilde{M}_T], \phi, \psi) \leq Inv( \tilde{M}_T , \phi, \psi) =k\]

The first equality and last inequality follow from \cite[2.23]{prest}. Clearly the above inequality gives us a contradiction. $\dagger_{\text{Claim}}$

Let $N_1, N_2$ be models of $\tilde{T}$. To show that $N_1\oplus N_2$ is a model of $\tilde{T}$, by \cite[2.18]{prest} it is enough to show that $Inv(N_1 \oplus N_2, \phi, \psi) = Inv(\tilde{M}_T, \phi, \psi)$ for every $\phi, \psi$ $pp$-formulas in one free
variable such that $\textbf{Th}_R \vdash \psi \to \phi$. Since $Inv(N_1 \oplus N_2, \phi, \psi) = Inv(N_1, \phi, \psi) 
Inv(N_2, \phi, \psi)$ (by \cite[2.23]{prest}), the result follows from the above claim.\end{proof}

\subsection{Superstability in classes closed under direct sums} In this section we will characterize superstability in classes of modules with pure embeddings closed under direct sums. Several examples of classes satisfying this hypothesis are given in \cite[3.10]{kuma}. 

\begin{remark}
Given $T$ a theory of modules, if $\Ka$ is closed under direct sums then $\Ka$ satisfies Hypothesis \ref{hyp} by \cite[3.8]{kuma}. Nevertheless, we keep Hypothesis \ref{hyp} as an assumption to make the presentation smoother.
\end{remark}

In \cite[\S 4]{kuma} limit models on classes of the form $\K^T$ were studied. Below we record the two assertions we will use in this paper.

\begin{fact}[{\cite[4.5]{kuma}}]\label{bigpi}
Assume $\lambda\geq |T|^+$. If $M$ is a $(\lambda,
\alpha)$-limit model in $\Ka$ and $\cof(\alpha)\geq |T|^+$, then $M$ is
pure-injective.
\end{fact}

\begin{fact}[{\cite[4.9]{kuma}}]\label{countablelim} Assume $\lambda\geq
|T|^+$ and $\K^T$ is closed under direct sums. If $M$ is
a $(\lambda, \omega)$-limit model and  $N$ is a
$(\lambda, |T|^+)$-limit model, then $M$ is isomorphic to $N^{(\aleph_0)}$.
\end{fact}

With this we are ready to obtain the next result.

\begin{lemma}\label{sigma-inj} Assume $\K^T$ is closed under direct sums.
If there exists $\mu \geq |T|^+$ such that $\K^T$ has uniqueness of limit models of cardinality $\mu$, then every $(\lambda, \alpha)$-limit model is $\Sigma$-pure-injective for every $\lambda \geq|T|$ and $\alpha < \lambda^+$ limit ordinal.
\end{lemma}
\begin{proof}
Let $M \in \K^T_\lambda$ be a $(\lambda, \alpha)$-limit model. Fix $N \in \K^{T}_\mu$ be a $(\mu, \omega)$-limit model  and $N^* \in \K^{T}_\mu$ be a $(\mu, |T|^+)$-limit model.  By Fact \ref{countablelim} we have that  $N \cong  (N^{*})^{(\omega)}$.

Then by uniqueness of limit models of size $\mu$ we have that $N \cong N^*$. Hence $(N^{*})^{(\omega)} \cong N^*$. Moreover, $N^*$ is pure-injective  by Fact \ref{bigpi}. Therefore, $N^*$ is $\Sigma$-pure-injective. 

Finally, observe that by Fact \ref{limeq} $N^* $ is elementary equivalent to $M$, hence $M$ is $\Sigma$-pure-injective by Fact \ref{pp-inj}. 
\end{proof}

Observe that in the above proof we did not use the full-strength of uniqueness of limit models of size $\mu$, but the weaker statement that the $(\mu, \omega)$-limit model is isomorphic to the $(\mu, |T|^+)$-limit model. We record it as a corollary for future reference.

\begin{cor}\label{cor-inj} Assume $\K^T$ is closed under direct sums.
If there exists $\mu \geq |T|^+$ such that the $(\mu, \omega)$-limit model is isomorphic to the $(\mu, |T|^+)$-limit model, then every $(\lambda, \alpha)$-limit model is $\Sigma$-pure-injective for every $\lambda \geq|T|$ and $\alpha < \lambda^+$ limit ordinal.
\end{cor}

\begin{lemma}\label{gstability}
Assume $\K^T$ is closed under direct sums.
If there exists $\mu \geq |T|^+$ such that $\K^T$ has uniqueness of limit models of cardinality $\mu$, then $\K^T$ is $\lambda$-stable for every $\lambda \geq |T|$.
\end{lemma}
\begin{proof}
Since $\Ka$ has uniqueness of limit models of size $\mu$, by Lemma \ref{sigma-inj} $\tilde{M}_T$ is $\Sigma$-pure-injective. Then by Fact \ref{sigma-stability}  $Th(\tilde{M}_T)= \tilde{T}$ is $\lambda$-stable for every $\lambda \geq |T|$. Therefore, it follows from Lemma \ref{stable} that $\K^T$ is $\lambda$-stable for every $\lambda \geq |T|$. 
 \end{proof}

The next result is easy to prove, but due to its importance we record it.

 \begin{cor}\label{easy}
If $M, N$ are limit models and $M, N$ are pure-injective, then $M$ is isomorphic to $N$.
\end{cor}
\begin{proof}
Since $M$,$N$ are limit models, it follows from Fact \ref{univ} that there are $f: M \to N$ and $g: N \to M$ pure embeddings. Then by Fact \ref{ipi} and the hypothesis that $M$, $N$ are pure-injective, we conclude that $M$ is isomorphic to $N$. 
\end{proof}

The next lemma is one of the key assertions of the section. In it we show that if the class is closed under direct sums, uniqueness of limit models in \emph{one} cardinal implies uniqueness of limit models in \emph{all} cardinals. 

\begin{lemma}\label{one-all}
Assume $\K^T$ is closed under direct sums. The following are equivalent.
\begin{enumerate}
\item For every $\lambda \geq  |T|$, $\K^T$ has uniqueness of limit models of cardinality $\lambda$.
\item $\Ka$ is superstable.
\item There exists $\lambda \geq |T|^+$ such that $\K^T$ has uniqueness of limit models of cardinality $\lambda$.
\end{enumerate}
\end{lemma}
\begin{proof}
(1) implies (2) and (2) implies (3) are clear, so we show (3) implies (1).

 Let $\lambda \geq |T|$. By Lemma \ref{gstability} it follows that $\K^T$ is $\lambda$-stable. Hence for every $\alpha < \lambda^+$ limit ordinal there is a $(\lambda, \alpha)$-limit model by Fact \ref{existence}. So we only need to show uniqueness of limit models. Let $M$ and $N$ two limit models of cardinality $\lambda$. By Lemma \ref{sigma-inj} we have that $M$ and $N$ are both pure-injective modules. Therefore, it follows from the above corollary that $M$ is isomorphic to  $N$. 
\end{proof}

We will give several additional equivalent conditions to the ones of Lemma \ref{one-all}, but before we do that  let us characterize superlimits in classes of modules.

\begin{lemma}
Assume $\Ka$ is $\lambda$-stable and $\lambda \geq |T|^+$. If $M$ is a superlimit of size $\lambda$, then $M$ is pure-injective. Moreover, if $\Ka$ is closed under direct sums, then $M$ is $\Sigma$-pure-injective.
\end{lemma}
\begin{proof}[Proof sketch]
 By Fact \ref{limits}.(3) $M$ is isomorphic to every $(\lambda, \alpha)$-limit model for $\alpha < \lambda^+$ limit ordinal. Then by Fact \ref{bigpi} $M$ is pure-injective. For the moreover part, observe that the existence of  a superlimit and the stability assumption imply uniqueness of limit models. Therefore, by Lemma \ref{sigma-inj}, $M$ is $\Sigma$-pure-injective.
\end{proof}

The following lemma can be proven using a similar technique to Fact \ref{limits}.(2) and using \cite[4.5]{kuma}.

\begin{lemma}\label{lim-gst}
If $M$ is a $(\lambda, |T|^+)$-limit model in $\Ka$, then $M$ is saturated in $\Ka$.
\end{lemma}

The following theorem characterizes superstability in classes of modules closed under direct sums.\footnote{Conditions (3) through (6) of the theorem below were motivated by \cite[1.3]{grva}.}

\begin{theorem}\label{equivalent}
Assume $\K^T$ is closed under direct sums. The following are equivalent.
\begin{enumerate}
\item $\Ka$ is superstable.
\item There exists $\lambda \geq |T|^+$ such that $\Ka$ has uniqueness of limit models of cardinality $\lambda$.
\item  For every $\lambda \geq |T|$, $\Ka$ has uniqueness of limit models of cardinality $\lambda$.
\item For every $\lambda \geq |T|^+$, $\Ka$ has a superlimit of cardinality $\lambda$.
\item  For every $\lambda \geq |T|$, $\K^T$ is $\lambda$-stable.
\item  For every $\lambda \geq |T|^+$, an increasing chain of $\lambda$-saturated models in $\Ka$ is $\lambda$-saturated in $\Ka$.
\end{enumerate}
\end{theorem}
\begin{proof}

(1) $\Leftrightarrow$ (2) $\Leftrightarrow$ (3) By Lemma \ref{one-all}.

(2) $\Rightarrow$ (5) By Lemma \ref{gstability}.

(5) $\Rightarrow$ (2) By Lemma \ref{stable} $\tilde{T}$ is $\lambda$-stable for every $\lambda\geq |T|$. Then by \cite[3.1]{prest} every model of $\tilde{T}$ is $\Sigma$-pure-injective. Let $\lambda = |T|^+$. By $\lambda$-stability and Fact \ref{existence}  there are $(\lambda, \alpha)$-limit models for every $\alpha < \lambda^+$ limit ordinal. The uniqueness of limit models of cardinality $\lambda$ follows from the fact that limit models are pure-injective, since they are models of $\tilde{T}$ by Corollary \ref{limeqq}, and by Corollary \ref{easy}.

(5) $\Rightarrow$ (6) Let $\{ M_i : i < \delta \} \subseteq \Ka$ be an increasing chain of $\lambda$-saturated models. By Lemma \ref{st=gst} every $M_i$ is a model of $\tilde{T}$ and $\lambda$-saturated in $\tilde{T}$. Moreover, by Lemma \ref{up-close}, for every $i < j$ we have that $M_i \preceq M_j$. Therefore, $\{ M_i : i < \delta \}$ is an increasing chain of $\lambda$-saturated models in $\tilde{T}$. 

Then by hypothesis and Lemma \ref{stable} $\tilde{T}$ is superstable as a first-order theory. Hence, by \cite{harnik}, $\bigcup_{i < \delta} M_i$ is a $\lambda$-saturated model of $\tilde{T}$. Therefore, by Lemma \ref{st=gst}, $\bigcup_{i < \delta} M_i$ is $\lambda$-saturated in $\Ka$.

(6) $\Rightarrow$ (2) Let $\lambda=2^{|T|}$ and $M$ be a $(2^{|T|}, |T|^+)$-limit model. By Fact \ref{bigpi} $M$ is pure-injective and by Lemma \ref{lim-gst} $M$ is saturated. Consider $\{ M^{n} : 0 < n < \omega \}$ an increasing chain in $\K_{2^{|T|}}$ where $M^{n}$ denotes $n$-many direct copies of $M$. Observe that each $M^n$ is pure-injective (because pure-injective modules are closed under finite direct sums) and that there are pure embeddings between $M^n$ and $M$ and vice versa by Fact \ref{univ}. Therefore, by Fact \ref{ipi}, $M \cong M^n$ for every $n > 0$. So in particular, $M^{n}$ is $\|M\|$-saturated for every $n > 0$.

 Then by hypothesis $M^{(\aleph_0)}$ is $\|M\|$-saturated, so $M^{(\aleph_0)}$ is saturated. Since saturated models of the same cardinality are isomorphic, $M^{(\aleph_0)} \cong M$. On the other hand, by Fact \ref{countablelim}, $M^{(\aleph_0)}$ is the $(2^{|T|}, \omega)$-limit model. Then by Corollary \ref{cor-inj} every limit model is $\Sigma$-pure-injective. Therefore, by Corollary \ref{easy}, $\Ka$ has uniqueness of limit models of size $2^{|T|}$. 

(4) $\Rightarrow$ (2) Similar to (5) implies (2) of Theorem \ref{main2}.

(3) $\Rightarrow$ (4) Let $\lambda \geq |T|^+$. By condition (5) $\Ka$ is $\lambda$-stable, so let $M$ be a  $(\lambda, |T|^+)$-limit model. By  Lemma \ref{lim-gst} $M$ is $\|M\|$-saturated. Moreover, by condition (6) any increasing chain of $\| M \|$-saturated models is $\| M \|$-saturated. Therefore it follows from Fact \ref{limits}.(3) that there is a superlimit of cardinality $\lambda$.

\end{proof}

\begin{remark}
In \cite[1.3]{grva} cardinals $\mu_{\ell}$ for $\ell \in\{1,...,7 \}$ were introduced. In the introduction of \cite{grva} it is asked if it is possible to calculate the values of the $\mu_\ell$'s for certain AECs. The above lemma shows that in classes of modules satisfying Hypothesis \ref{hyp} and closed under direct sums  $\mu_3=\mu_7=|T|$ and $\mu_4, \mu_5 \leq |T|^+$. We did not calculate the values of $\mu_1, \mu_2$ and $\mu_6$ since they measure properties that are more technical than the ones presented above and which we have not introduced. We hope to study those properties in future work. 
\end{remark}

\begin{remark} Let $T$ be a theory of modules such that Hypothesis \ref{hyp} holds. Then by Fact \ref{pp=gtp}  $\Ka$ is $(< \aleph_0)$-tame and by hypothesis $\Ka$ has amalgamation, joint embedding and no maximal models. Therefore, we could have simply quoted \cite[1.3]{grva} to obtain (4),(5),(6) imply (2) of the above theorem. We decided to provide the proofs for those directions to make the paper more transparent and since the proofs in our case are easier than in the general case. An important difference between our methods and those of \cite{grva} and \cite{vaseyt} is that the results of Grossberg and Vasey do not use the hypothesis that the class is closed under direct sums, but only Hypothesis \ref{hyp}. We will come back to this in Subsection 4.4.
\end{remark}

\subsection{Algebraic characterizations of superstability and pure-semisimple rings} We begin by giving algebraic characterizations of  superstability  in classes of modules closed under direct sums.

\begin{theorem}\label{alge} Assume $\K^T$  is closed under direct sums. The following are equivalent.
\begin{enumerate}
\item $\Ka$ is superstable.
\item There exists $\lambda \geq |T|^+$ such that $\Ka$ has uniqueness of limit models of cardinality $\lambda$.
\skipitems{4}
\item Every $M\in \Ka$ is pure-injective.
\item There exists $\lambda \geq |T|^+$ such that $\Ka$ has a $\Sigma$-pure-injective universal model in $\Ka_\lambda$.
\item Every limit model in $\Ka$ is $\Sigma$-pure-injective. 
\end{enumerate}

\end{theorem}
\begin{proof}
(1) $\Rightarrow$ (2) By Theorem \ref{equivalent}.

(2) $\Rightarrow$ (9) By Lemma \ref{sigma-inj}.

(9) $\Rightarrow$ (7) Let $M$ be an $R$-module and let $\mu = \| M \|^{|R|+\aleph_0}$.

By \cite[3.16]{kuma} $\Ka$ is $\mu$-stable, so let $N$ be a $(\mu, \omega)$-limit model such that $M \leq_{pp} N$. Then by hypothesis $N$ is $\Sigma$-pure-injective. As $M \leq_{pp} N$, by Fact \ref{pp-inj} it follows that $M$ is $\Sigma$-pure-injective. Hence $M$ is pure-injective.

(7) $\Rightarrow$ (8) Let $\lambda = 2^{|T|}$, then by \cite[3.16]{kuma} $\Ka$ is $\lambda$-stable. Let $M$ be a $(\lambda, \omega)$-limit model. $M$ is universal in $\Ka_\lambda$ by Fact \ref{univ} and $M$ is $\Sigma$-pure-injective by condition (7) and closure under direct sums.

(8) $\Rightarrow$ (2) Let $\lambda \geq |T|^+$ and $M$ be a $\Sigma$-pure-injective universal model in $\Ka_\lambda$. As in (6) to (2) of Theorem \ref{equivalent}, consider $\{ M^{n} : 0 < n  < \omega \}$ an increasing chain in $\Ka_\lambda$. Observe that $M^{n+1}$ is universal over $M^{n}$ for every $n > 0 $ by \cite[4.8]{kuma}. Hence the chain witnesses that $M^{(\aleph_0)}$ is a $(\lambda, \omega)$-limit model in $\Ka$. So $\Ka$ is $\lambda$-stable by Fact \ref{existence}.  We are left to show uniqueness of limit models of size $\lambda$. Let $N_1$ and $N_2$ be two limit models of cardinality $\lambda$. Since $M^{(\aleph_0)}$, $N_1$ and $N_2$ are elementary equivalent by Fact \ref{limeq} and $M$ is $\Sigma$-pure-injective by hypothesis, then $N_1$ and $N_2$ are pure-injective by Fact \ref{pp-inj}. So from Corollary \ref{easy} we conclude that $N_1$ is isomorphic to $N_2$. Therefore, $\Ka$ has uniqueness of limit models of cardinality $\lambda$.

\end{proof}

\begin{remark}
In  \cite[4.10]{kuma}, stability is characterized by the existence of a pure-injetive universal model. The equivalence between (1) and (8) is another instance of how the existence of a universal model can be used to characterize an stability-like property. 
\end{remark}

 The next theorem gives many equivalent conditions to the notion of pure-semisimple ring and relates it to superstability. It follows directly from Theorem \ref{equivalent} and Theorem \ref{alge} as a ring $R$ is pure-semisimple if and only if every $M \in \K^{\textbf{Th}_R}$ is pure-injective (Definition \ref{pss}).\footnote{Conditions (4) through (7) of the theorem below were motivated by \cite[1.3]{grva}.}

\begin{theorem}\label{main}
For a ring $R$ the following are equivalent.
\begin{enumerate}
\item $R$ is left pure-semisimple.
\item $\K^{\textbf{Th}_R}$ is superstable.
\item There exists $\lambda \geq (|R| + \aleph_0)^+$ such that $\K^{\textbf{Th}_R}$  has uniqueness of limit models of cardinality $\lambda$.
\item  For every $\lambda \geq |R| + \aleph_0$, $\K^{\textbf{Th}_R}$  has uniqueness of limit models of cardinality $\lambda$.
\item For every $\lambda \geq (|R| + \aleph_0)^+$, $\K^{\textbf{Th}_R}$  has a superlimit of cardinality $\lambda$.
\item  For every $\lambda \geq |R| + \aleph_0$, $\K^{\textbf{Th}_R}$  is $\lambda$-stable.
\item  For every $\lambda \geq (|R| + \aleph_0)^+$, an increasing chain of $\lambda$-saturated models in $\K^{\textbf{Th}_R}$  is $\lambda$-saturated in $\K^{\textbf{Th}_R}$ .
\item There exists $\lambda \geq  (|R| + \aleph_0)^+$ such that $\K^{\textbf{Th}_R}$ has a $\Sigma$-pure-injective universal model in $\K^{\textbf{Th}_R}_\lambda$.
\item Every limit model in  $\K^{\textbf{Th}_R}$  is $\Sigma$-pure-injective. 
\end{enumerate}
\end{theorem}

 As it was mentioned in the introduction, Shelah noticed that superstability of the class of modules and pure-semisimple rings are equivalent in Theorem 8.7 of \cite{shlaz} (without mentioning pure-semisimple rings). The precise equivalence Shelah noticed is similar to the equivalence between (1) and (6) of the above theorem. Shelah does not prove that pure-semisimplicty implies superstability of the theory of modules ((2) to (1) of his Theorem 8.7). The notion of superstability studied in \cite[\S 8]{shlaz} is not first-order superstability, but superstability with respect to positive primitive formulas. This is equivalent to our notion of superstability in the class of modules by Fact \ref{pp=gtp}. For the direction that Shelah provides a proof ((1) to (5) of his Theorem 8.7), his technique is different from ours as he proceeds by contradiction and builds a tree of formulas.

One can add one more equivalent condition to the above theorem.

\begin{lemma}
For a ring $R$ the following are equivalent.
\begin{enumerate}
\item $R$ is left pure-semisimple.
\skipitems{1}
\item There exists $\lambda \geq (|R| + \aleph_0)^+$ such that $\K^{\textbf{Th}_R}$ has uniqueness of limit models of cardinality $\lambda$.
\skipitems{6}
\item For all theories of modules $T$ and for all $\lambda \geq |R|+\aleph_0$, if $\K^T$ satisfies Hypothesis \ref{hyp} and it is closed under direct sums, then $\K^T$ has uniqueness of limit models of cardinality $\lambda$.
\end{enumerate}
\end{lemma}
\begin{proof}
$(1)$ implies $(10)$ follows from condition $(7)$ of Theorem \ref{alge}. Moreover, it is clear that $(10)$ implies $(3)$.
\end{proof}

Assuming that the ring is left coherent, Theorem \ref{alge} can further be used to obtain a new characterization of left noetherian rings.

\begin{lemma}\label{absol}
Let $R$ be a left coherent ring and $K_{\text{abs}}$ be the class of absolutely pure left $R$-modules. $R$ is left noetherian if and only if $\K_{A}=(K_{\text{abs}}, \leq)$ is superstable.
\end{lemma}
\begin{proof} 
Since $R$ is left coherent, by \cite[3.4.24]{prest09}, there is $T$ a first-order theory such that $Mod(T)=K_{\text{abs}}$. Observe that for absolutely pure modules embeddings are the same as pure embeddings. Moreover, absolutely pure modules are closed under direct sums and because of it have joint embedding and amalgamation with respect to pure embeddings by \cite[3.8]{kuma}. Therefore, we can use the results of Theorem \ref{alge}. More precisely, we will show the equivalence with condition (7) of Theorem \ref{alge}, i.e., we will show that $R$ is left noetherian if and only if every absolutely pure module is pure-injective.

If $R$ is left noetherian, then every absolutely pure module is injective by Fact \ref{equivnoe}.(5). Hence every absolutely pure module is pure-injective.

 If every absolutely pure module is pure-injective, then it is clearly injective. Thus $R$ is left noetherian by  Fact \ref{equivnoe}.(5). 
\end{proof}

\begin{remark}
The above assertion is still true without the assumption that $R$ is left coherent. That case lies outside of the scope of the present paper as in that case the class of absolutely pure modules is not first-order axiomatizable. That result is presented in \cite{maz2} .
\end{remark}

\begin{remark}
Assuming that the ring is right coherent, Theorem \ref{alge} can also be used to characterize left perfect rings. The proof is similar to that of Lemma \ref{absol}, but we do not present it as we have not introduced many of the notions needed to setup the proof. Moreover, a proof without the assumption that the ring is right coherent is given in \cite[3.15]{maz1}. That case is significantly more intricate as the class of flat modules is not first-order axiomatizable and is not closed under pure-injective envelopes.
\end{remark}

We think that Theorem \ref{main2}, Theorem \ref{main} and Lemma \ref{absol} hint to the the fact that limit models and superstability could shed light in the understanding of algebraic concepts. They also provide further evidence of the naturality of the notion of superstability.

\subsection{Superstable classes} In this section we will characterize superstability in classes of modules without assuming that the class is closed under direct sums. As in previous subsections we assume Hypothesis \ref{hyp}. In this section we assume the reader has some familiarity with first-order model theory.

Recall that given $T'$ a complete first-order theory, $\kappa(T')$ is the least cardinal such that every type of $T'$ does not fork over a set of size less than $\kappa(T')$. In this case, nonforking refers to first-order nonforking. Since it is well-known that $\kappa(T') \leq |T'|^+$ for stable theories, the following improves \cite[4.7]{kuma}.

\begin{lemma}\label{lim-iso} Let $\lambda \geq |T|$ and $\alpha, \beta < \lambda^+$ be limit ordinals.
If $M$ is a $(\lambda, \alpha)$-limit model in $\Ka$, $N$ is a $(\lambda, \beta)$-limit model in $\Ka$ and $\cof(\alpha), \cof(\beta)\geq \kappa(\tilde{T})$, then $M$ is isomorphic to $N$.
\end{lemma}
\begin{proof}
By Lemma \ref{lim-til} $M$ is a $(\lambda, \alpha)$-limit model in  $(Mod(\tilde{T}), \preceq)$ and $N$ is a $(\lambda, \beta)$-limit model in $(Mod(\tilde{T}), \preceq)$.
 Then it follows from \cite[1.6]{grvavi} that $M$ and $N$ are both saturated models of cardinality $\lambda$ in $\tilde{T}$. Therefore, we conclude that $M$ is isomorphic to $N$.
\end{proof}

Recall the following notions introduced in \cite[\S 2]{vaseyt}. Given an AEC $\K$ and $\mu$ a cardinal,  $\text{Stab}(\K) = \{ \mu : \K \text{ is $\mu$-stable} \}$ and  $\underline{\kappa}(\K_\mu, \lea^{\text{univ}})$ is the set of regular cardinals $\chi$ such that whenever $\{ M_i : i < \chi \}$ is an increasing chain in $\K_\mu$ with $M_{i +1}$ is universal over $M_i$ for every $i < \chi$ and $p \in \gS(\bigcup_{i < \chi} M_i)$, then there is $i < \chi$ such that $p$ does not split over $M_i$. Since we will not use the notion of splitting, we will not introduce it. It  is somehow similar to first-order splitting, the definition is presented in \cite[2.3]{vaseyt}.

In \cite[4.12]{kuma} it is asked to describe the spectrum of limit models for classes of the form $\K^T$ where $T$ is a theory of modules. The case when the ring is countable is studied in \cite[4.12]{kuma}. The next result provides a partial solution to that question.

\begin{theorem}\label{lim-big}
Let $\lambda \geq |T|^+$ be a regular cardinal.
Let $M$ be a $(\lambda, \alpha)$-limit model in $\Ka$, then:
\begin{enumerate}
\item If $\cof(\alpha) \geq \kappa(\tilde{T})$, then $M$ is isomorphic to the $(\lambda, \lambda)$-limit model.
\item If $\cof(\alpha) < \kappa(\tilde{T})$, then $M$ is not isomorphic to the $(\lambda, \lambda)$-limit model.
\end{enumerate}
\end{theorem}
\begin{proof}
(1) follows from Lemma \ref{lim-iso} and the fact that $\kappa(\tilde{T}) \leq |T|^+$. So we prove (2). Assume for the sake of contradiction that $M$ is isomorphic to the $(\lambda, \lambda)$-limit model in $\Ka$ and let $\tilde{\K}:= (Mod(\tilde{T}), \preceq)$.

Since $\cof(\alpha) < \kappa(\tilde{T})$, by \cite[4.8]{vaseyt} $\cof(\alpha) \notin \underline{\chi}(\tilde{\K}) = \bigcup_{\mu \in \text{Stab}(\tilde{\K})} \underline{\kappa}(\tilde{\K}_\mu, \leq_{\tilde{\K}}^{\text{univ}})$. Since $M$ is a limit model in $\Ka$, by Fact \ref{existence}, $\Ka$ is $\lambda$-stable so by Lemma \ref{stable} $\tilde{\K}$ is $\lambda$-stable. Hence $\cof(\alpha) \notin  \underline{\kappa}(\tilde{\K}_\lambda, \leq_{\tilde{\K}}^{\text{univ}})$. Then by definition of $\underline{\kappa}$ there is $\{L_i : i < \cof(\alpha) \}$ an increasing chain in $\tilde{\K}_\lambda$ with $L_{i+1}$ universal over $L_i$ and $p \in \gS(\bigcup_{i < \cof(\alpha)} L_i)$ such that $p$ splits over $L_i$ in $\tilde{\K}$ for every $i < \cof(\alpha)$.

Observe that $L=\bigcup_{i < \cof(\alpha)} L_i$ is a $(\lambda, \cof(\alpha))$-limit model in $\K^T$ by Lemma \ref{lim-til}. Then doing a back-and-forth argument $L \cong M$. And since by hypothesis $M$ is isomorphic to the $(\lambda, \lambda)$-limit model, $L$ is isomorphic to it. By Fact \ref{limits}.(1) it follows that $L$ is a $\lambda$-saturated model in $\Ka$. So by Lemma \ref{st=gst} $L$ is $\lambda$-saturated in $\tilde{\K}$. Then by \cite[4.12]{vaseyt} there is $i < \cof(\alpha)$ such that $p$ does not split over $L_i$ in $\tilde{\K}$. This contradicts the choice of the $L_i's$. 
\end{proof}

Moreover, the result of the above theorem gives a positive solution above $|T|^+$ to Conjecture 2 of \cite{bovan} in the case when $T$ is a theory of modules satisfying Hypothesis \ref{hyp}. 

\begin{cor}
Let $T$ be a theory of modules such that  $\K^{T}$
satisfies Hypothesis \ref{hyp}. Assume that $\lambda\geq |T|^+$ is a regular cardinal such that $\K^T$ is $\lambda$-stable. Then  \[\Delta_\lambda:=\{ \alpha < \lambda^+ : \cof(\alpha)=\alpha \text{ and the } (\lambda, \alpha)\text{-limit model is isomorphic to the } (\lambda,\lambda)\text{-limit model}\}\] is an end segment of regular cardinals.
\end{cor}

We finish this section by presenting a similar result to Theorem \ref{equivalent}, but without the assumption that $T$ is closed under direct sums.

\begin{theorem}\label{equivalent2}
Assume $\Ka$ satisfies Hypothesis \ref{hyp}. The following are equivalent.
\begin{enumerate}
\item $\Ka$ is superstable.
\item  For every $\lambda \geq 2^{|T|}$, $\Ka$ has uniqueness of limit models of cardinality $\lambda$.
\item For every $\lambda \geq 2^{|T|}$, $\Ka$ has a superlimit of cardinality $\lambda$.
\item  For every $\lambda \geq 2^{|T|}$, $\K^T$ is $\lambda$-stable.
\item  For every $\lambda \geq 2^{|T|}$, an increasing chain of $\lambda$-saturated models in $\Ka$ is $\lambda$-saturated in $\Ka$.
\end{enumerate}
\end{theorem}
\begin{proof}[Proof sketch]
(1) $\Rightarrow$ (4) By (1) and Lemma \ref{lim-til} $(Mod(\tilde{T}), \preceq)$ has uniqueness of limit models in a tail of cardinals. Then by Fact \ref{existence} $\tilde{T}$ is stable in a tail of cardinals. Since $\tilde{T}$ is a first-order theory, the tail has to begin at most in $2^{|T|}$. 

(4) $\Rightarrow$ (2) By Lemma \ref{stable} $\tilde{T}$ is superstable, then by \cite{shbook} (see \cite[1.1.(3)]{grva}) $\kappa(\tilde{T}) = \aleph_0$. The result follows from Fact \ref{existence} and Lemma \ref{lim-iso}.

(2) $\Rightarrow$ (1) Clear.

(4) $\Rightarrow$ (5)  Similar to (5) implies (6) of Theorem \ref{equivalent}.

(5) $\Rightarrow$ (4) The idea is to prove that $\tilde{T}$ is superstable and then by Lemma \ref{stable} the result would follow. To prove that $\tilde{T}$ is superstable, by \cite{shbook} (see \cite[1.1.(2)]{grva}),  it is enough to show that an elementary increasing chain of $\lambda$-saturated models in $\tilde{T}$ is $\lambda$-saturated. The proof is similar to (5) implies (6) of Theorem \ref{main} by using Lemma \ref{st=gst}.

(2) $\Rightarrow$ (3)  Similar to (3) implies (4) of Theorem \ref{equivalent}.

(3) $\Rightarrow$ (4) Assume for the sake of contradiction that (4) fails, then by Lemma \ref{stable} $\tilde{T}$ is not superstable.  Then by \cite{shbook} (see \cite[1.1.(3)]{grva}) $\kappa(\tilde{T}) > \aleph_0$. Let $\lambda= (2^{|T|})^{+}$. Observe that $\Ka$ is $\lambda$-stable since $\lambda^{|T|}=\lambda$, so by Fact \ref{limits}.(3) and (3) $\Ka$ has uniqueness of limit models in $\lambda$. Now, by Theorem \ref{lim-big}.(2), we know that the $(\lambda, \omega)$-limit model is not isomorphic to the $(\lambda, \lambda)$-limit model, which clearly gives us a contradiction. \end{proof}

\begin{remark}
Compared to \cite[1.3]{grva}, the above theorem improves the bounds where the \emph{nice properties} show up from $\beth_{(2^{|T|})^+}$ to $2^{|T|}$ in the case of classes of modules satisfying Hypothesis \ref{hyp}. It is worth pointing out that to obtain (3), (5) imply (1) we could have simply quoted \cite[1.3]{grva}. We decided to provide the proofs for those direction to make the paper more transparent and to show the deep connection between $\Ka$ and $\tilde{T}$. 
\end{remark}

Besides the difference in the bounds between Theorem \ref{equivalent} and Theorem \ref{equivalent2}. The techniques are also quite different, while the proof of the first theorem relies more on algebraic notions, the proof of the second theorem relies heavily on model theoretic methods.


\begin{thebibliography}{She01b}



\bibitem[Aus74]{auslander1}
Maurice Auslander,
\emph{Representation theory of artin algebras II}
Comm. Algebra \textbf{1} (1974), 269 -- 310.

\bibitem[Aus76]{auslander}
Maurice Auslander, \emph{Large modules over Artinian algebras} in Algebra topology and category theory, Academic Press (1976), 1--16.

\bibitem[BaMc82]{balmc}
John T. Baldwin and Ralph N. McKenzie, \emph{Counting models in universal {H}orn classes}, Algebra Universalis \textbf{15} (1982), 359--384.


\bibitem[Bal09]{baldwinbook09}
John Baldwin, \emph{Categoricity}, American Mathematical Society
(2009).


\bibitem[BCG+]{grp}
John Baldwin, Wesley Calvert, John Goodrick, Andres Villaveces and
Agatha Walczak-Typke, \emph{Abelian groups as aec's}. Preprint. URL: 
www.aimath.org/WWN/categoricity/abeliangroups\_10\_1\_3.tex.

\bibitem[BET07]{baldwine}
 John Baldwin, Paul Eklof and Jan Trlifaj, \emph{As an abstract
elementary class}, Annals of Pure and Applied Logic
\textbf{149}(2007), no. 1,25--39.

\bibitem[Bon14a]{extendingframes}
Will Boney, \emph{Tameness and extending frames}, Journal of Mathematical Logic \textbf{14} (2014), no. 2, 1450007, 27 pp.

\bibitem[BoVan]{bovan}
Will Boney and Monica VanDieren, \emph{Limit Models in Strictly Stable Abstract Elementary Classes}, Preprint. URL: https://arxiv.org/abs/1508.04717 .

\bibitem[Bum65]{bumby}
Richard T. Bumby, \emph{Modules which are isomorphic to submodules of
each other}, Archiv der Mathematik \textbf{16}(1965), 184--185.

\bibitem[Cha60]{cha}
S.U. Chase, \emph{Direct products of modules}
Trans. Amer. Math. Soc. \textbf{97} (1960), 457 -- 473.



\bibitem[Dru13]{dru}
Fred Drueck, \emph{Limit models, superlimit models and two cardinal problems in abstract elementary classes}, Ph. D. thesis, 2013. URL: http://homepages.math.uic.edu/~drueck/thesis.pdf 



\bibitem[Ekl71]{eklof}
Paul Eklof, \emph{ Homogeneous universal modules}, Mathematica
Scandinavica \textbf{29}(1971), 187--196. 

\bibitem[Fai66]{faith}
Carl Faith, \emph{Rings with ascending condition on annihilators},
 Nagoya Math. J. \textbf{27}(1966), no. 1, 179--191.



\bibitem[Gro1X]{ramibook}
Rami Grossberg, \emph{ A Course in Model Theory}, in Preparation,
201X.

\bibitem[Gro02]{grossberg2002}
Rami Grossberg, \emph{Classification theory for abstract elementary
classes},
  Logic and Algebra (Yi~Zhang, ed.), vol. 302, American Mathematical
Society,
  2002, 165--204.

\bibitem[GrSh86]{grsh}
Rami Grossberg and Saharon Shelah, \emph{A nonstructure theorem for an infinitary theory which  has  the  unsuperstability  property}, Illinois  Journal  of  Mathematics \textbf{30} (1986),no. 2, 364 --390.

\bibitem[GrVan06]{tamenessone}
Rami Grossberg and Monica VanDieren, \emph{Galois-stability for tame
abstract elementary classes}, Journal
  of Mathematical Logic \textbf{6} (2006), no.~1, 25--49.

\bibitem[GVV16]{grvavi}
Rami Grossberg, Monica VanDieren and Andres Villaveces, \emph{Uniqueness of limit models in classes with amalgamation}, Mathematical Logic Quarterly \textbf{62} (2016), 367--382.


\bibitem[GrVas17]{grva}
Rami Grossberg and Sebastien Vasey, \emph{Equivalent definitions of
superstability in tame abstract elementary classes}, The Journal of
Symbolic Logic \textbf{82} (2017), no. 4, 1387 -- 1408.

\bibitem[GKS18]{gks}
Pedro A. Guil Asensio, Berke Kalebogaz and Ashish K. Srivastava,
\emph{The Schr\"{o}der-Bernstein problem for modules}, 
Journal of Algebra \textbf{498} (2018),
153--164.


\bibitem[Har75]{harnik}
Victor Harnik, \emph{On the existence of saturated models of stable theories}, Proceedings of the American Mathematical Society \textbf{52} (1975), 361--367.


\bibitem[Hui00]{hui}
 Birge Huisgen-Zimmermann, \emph{Purity, algebraic compactness, direct sum decompositions, and representation type}, Infinite length modules (Bielefeld, 1998), 331--367, Trends Math., Birkhäuser, Basel, 2000.

\bibitem[KolSh96]{kosh}
Oren Kolman and Saharon Shelah, \emph{Categoricity of Theories in
$L_{\omega, \kappa}$ when $\kappa$ is a measurable cardinal. Part 1},
Fundamenta Mathematicae \textbf{151} (1996), 209--240.

\bibitem[KuMa]{kuma}
Thomas G. Kucera and Marcos Mazari-Armida, \emph{On universal modules with pure embeddings}, Accepted by Mathematical Logic Quarterly. URL: https://arxiv.org/abs/1903.00414


\bibitem[LRV]{lrv}
Michael Lieberman, Ji\v{r}\'{\i} Rosick\'{y}, and Sebastien Vasey, \emph{Weak factorization systems and stable independence}, Preprint. URL: https://arxiv.org/abs/1904.05691v2


\bibitem[Maz20]{maz}
 Marcos Mazari-Armida, \emph{Algebraic description of limit models in classes of abelian groups}, Annals of Pure and Applied Logic \textbf{171} (2020), no. 1, 102723. 
 
 \bibitem[Maz21]{maztor}
Marcos Mazari-Armida, \emph{A model theoretic solution to a problem of L\'{a}szl\'{o} Fuchs}, Accepted by Journal of Algebra, Journal of Algebra \textbf{567} (2021), 196--209.

\bibitem[Maz]{maz1}
Marcos Mazari-Armida, \emph{On superstability in the class of flat modules and perfect rings}, Accepted by Proceedings of AMS, 14 pages. URL: https://arxiv.org/abs/1910.08389


\bibitem[Maz1]{maz2}
Marcos Mazari-Armida, \emph{Some stable non-elementary classes of modules}, preprint, 21 pages. URL: https://arxiv.org/abs/2010.02918


\bibitem[MaVa18]{mv}
Marcos Mazari-Armida and Sebastien Vasey, \emph{Universal classes near $\aleph_1$}, The Journal of Symbolic Logic \textbf{83} (2018), no. 4, 1633–1643

\bibitem[Noe21]{noe}
 Emmy Noether, \emph{Idealtheorie in Ringbereichen}, Math. Ann. \textbf{83} (1921), 24 -- 66.


\bibitem[Pre84]{prest84}
Mike Prest, \emph{Rings of finite representation type and modules of finite Morley rank},
Journal of Algebra  \textbf{88} Issue 2 (1984), 502--533.

\bibitem[Pre88]{prest}
Mike Prest, \emph{Model Theory and Modules}, London Mathematical
Society Lecture Notes Series Vol. 130, Cambridge University Press,
Cambridge (1988).


\bibitem[Pre09]{prest09}
Mike Prest, \emph{ Purity, Spectra and Localisation} , Encyclopedia of Mathematics and its Applications, Vol. 121, Cambridge University Press  (2009) (xxviii+769pp). 


\bibitem[Sim77]{simson77}
Daniel Simson,
\emph{Pure semisimple categories and rings of finite representation type},
Journal of Algebra \textbf{48} vol. 2
(1977),  290--302; Corrigendum \textbf{67} (1980), 254 -- 256.

\bibitem[Sim81]{simson81}
Daniel Simson, \emph{Partial Coxeter functors and right pure semisimple hereditary rings},
Journal of Algebra \textbf{71} (1981), Issue 1, 195--218.

\bibitem[Sim00]{simson00}
 Daniel Simson, \emph{An Artin problem for division ring extensions and the pure semisimplicity conjecture. II}, Journal of Algebra \textbf{227} (2000), no. 2, 670--705.


\bibitem[Sh75]{shlaz}
Saharon Shelah, \emph{The lazy model-theoretician's guide to stability},  Logique Et Analyse \textbf{18} (1975), no. 71/72,  241--308. 

\bibitem[Sh87a]{sh88}
Saharon Shelah, \emph{Classification of nonelementary classes, {II}.
{A}bstract
  elementary classes}, Classification theory (John Baldwin, ed.)
(1987), 419--497.

\bibitem[Sh87b]{sh300}
Saharon Shelah, \emph{Universal classes}, Classification theory (John
Baldwin, ed.) (1987), 264--418.



\bibitem[Sh99]{sh394}
Saharon Shelah, \emph{Categoricity  for  abstract  classes  with  amalgamation},  Annals  of  Pure  andApplied Logic \textbf{98} (1999), no. 1, 261 -- 294.


\bibitem[Sh78]{shbook}
Saharon Shelah, \emph{Classification Theory and the Number of Non-Isomorphic Models}, North-Holland Publishing Company (1978).

\bibitem[Sh:h]{shelahaecbook}
Saharon Shelah, \emph{Classification Theory for Abstract Elementary Classes},
 vol. 1 \& 2, Mathematical Logic and Foundations, no. 18 \& 20, College
  Publications (2009).


\bibitem[ShVi99]{shvi}
Saharon Shelah and Andres Villaveces, \emph{Toward categoricity for
classes with no maximal models}, Annals of Pure and Applied Logic
\textbf{97} (1-3):1-25 (1999)

\bibitem[Tar54]{tarski} Alfred Tarski, \emph{Contributions to the theory of models {I}}, Indagationes Mathematicae \textbf{16} (1954), 572--581.

 \bibitem[Van06]{van06}
Monica VanDieren, \emph{Categoricity in abstract elementary classes with no maximal models}, Annals Pure Applied Logic \textbf{141} (2006), 108--147.

 \bibitem[Van16]{vand}
 Monica VanDieren, \emph{Superstability and symmetry}, Annals of Pure and Applied Logic \textbf{167} (2016), no.  12, 1171--1183.


\bibitem[Vas17]{vaseyu}
Sebastien Vasey,\emph{ Shelah's eventual categoricity conjecture in universal classes: part I}, Annals of Pure and Applied Logic \textbf{168} (2017), no. 9, 1609 - 1642. 

\bibitem[Vas18]{vaseyt}
Sebastien Vasey, \emph{Toward a stability theory of tame abstract elementary classes}, Journal of Mathematical Logic \textbf{18} (2018), no. 2, 1850009.


\bibitem[Vas19]{vasey18}
Sebastien Vasey, \emph{The categoricity spectrum of large abstract elementary classes}, Selecta Mathematica \textbf{25} (2019), no. 5, 65.


  \bibitem[Zie84]{ziegler}
Martin Ziegler, \emph{Model Theory of Modules}, Annals of Pure and
Applied Logic \textbf{26} (1984), 149 -- 213.

  \bibitem[Zim79]{zim}
Birge Zimmermann-Huisgen, \emph{Rings whose right modules are direct sums of indecomposable modules},
Proceedings of the American Mathematical Society \textbf{77} (2) (1979), 191--197

 \bibitem[Zim79]{zimm}
Wolfgang Zimmermann, \emph{Rein injektive direkte summen von moduln},
Communications in Algebra \textbf{5} (10) (1977), 1083 -- 1117. 

\end{thebibliography}

\end{document}